\documentclass{article}

\usepackage[active]{srcltx}
\usepackage{setspace}
\usepackage[USenglish]{babel} %francais, polish, spanish, ...
\usepackage[T1]{fontenc}
\usepackage[ansinew]{inputenc}
\usepackage{lmodern}
\usepackage{graphicx}
\usepackage{amsmath}
\usepackage{amsthm}
\usepackage{amsfonts,amssymb,latexsym}
\usepackage{color}
\usepackage{float}
\usepackage{enumerate}
\usepackage{esint}
\usepackage{epstopdf}
\usepackage[margin=1.2in]{geometry}

\newtheorem{thm}{Theorem}[section]
\newtheorem{lemma}[thm]{Lemma}
\newtheorem{prop}[thm]{Proposition}
\newtheorem{defn}[thm]{Definition}

\def\Cpw{\ensuremath{\mathcal{C}_{pw}([0,T])}}

\def\L{\ensuremath{\mathcal{L}}}

\def\half{\ensuremath{\frac{1}{2}}}

\def\n{\ensuremath{\mathbb{N}}}

\def\re{\ensuremath{\mathbb{R}}}

\newcommand{\lf}[1]{#1}
\newcommand{\lfff}[1]{#1}
\newcommand{\lff}[1]{#1}
\newcommand{\short}[1]{#1}
\newcommand{\update}[1]{#1}
\newcommand{\dt}[1]{\frac{d{#1}}{dt}}

\title{Optimal strategies for operating energy storage in an arbitrage market\thanks{This work was
supported by EPSRC (grant EP/K002228/1) and the Alfred P. Sloan Foundation New York.}} 

\author{Lisa Flatley\thanks{Mathematics Institute, University of Warwick, Gibbet Hill Road, Coventry, CV4 7AL.  Corresponding email address: L.Flatley@warwick.ac.uk} \and 
Robert S MacKay\footnotemark[2] \and Michael Waterson \thanks{Economics Department, University of Warwick, Gibbet Hill Road, Coventry, CV4 7AL.}}
\date{}

\begin{document}
\maketitle

\begin{abstract}
We characterize profit-maximizing operating strategies, over some time horizon $[0,T],$ for an energy store which is trading in an arbitrage market.  Our theory allows for leakage, operating inefficiencies, operating constraints and general cost functions.  In the special case where the operating cost of a store depends only on its instantaneous power output (or input), we present an algorithm to determine the optimal strategies.  A key feature is that this algorithm is localized in time, in the sense that the action of the store at a time $t\in[0,T]$ only requires information about electricity prices over some subinterval of time $[t,t_k]\subset[t,T].$  \lff{To introduce more} complex storage models, we discuss methods \lff{for} an example which includes minimum switching times between modes of operation.%\footnote{This paper is available, in a previous form, on the University of Warwick Economics working papers site and the members-only site associated with the EPSRC-funded project IMAGES.}
\end{abstract}

%\begin{keywords}\textnormal{Dynamic optimization, energy storage, optimal localized algorithm.}\end{keywords}

%\begin{AMS}\textnormal{49K30, 93C10}\end{AMS}

\pagestyle{myheadings}
\thispagestyle{plain}
%\markboth{OPTIMAL STORAGE OPERATION}{OPTIMAL STORAGE OPERATION}

\section{Introduction}
Over the coming decades, the UK energy market faces significant challenges as it strives to meet its climate change targets.  Low carbon and renewable generation will need to play a more dominant role in our future electricity supply market.  However, renewable energy is intermittent and its availability is driven by uncontrollable elements, such as wind speed and solar intensity.  Other options, such as nuclear power and Carbon Capture and Storage, are generally considered to be less flexible than traditional thermal plants.  The problems are clear.  The unreliability of renewable supply means that there will always be a need for a quick-reacting back-up, in order to ensure that demand is met.  On the other hand, at off-peak demand times, supply may be curtailed (\short{often at a high cost}) if the system is not flexible enough to respond.  %Curtailment often comes at a high cost: in April 2011, for example, National Grid \lf{paid wind farms} up to $\pounds$800/MWh in order to reduce electricity production from Scottish wind farms (compared to $\pounds$28/MWh for the cheapest coal power station not to run).  Flexibility of supply, therefore, is key to the successful and cost-effective integration of renewable generation into the energy system.

On the demand side, despite improved energy efficiency measures, we can expect to see an overall increase in consumption as the electrification of transport and heating become more commonplace \cite{DECC}.  \lf{If all environmental targets are met on time, then National Grid's Gone Green scenario predicts a 12.6$\%$ increase in Great Britain's peak demand by 2035, from 2013 levels \cite{NG}.}  In addition, future demand profiles are likely to look very different from today.  Electric heat pumps, for example, will create a marked seasonal peaking of demand during the colder, winter months.  Electric vehicle charging \short{and new smart technologies}, on the other hand, \short{could} distort our daily demand patterns.  %Being a relatively new technology, it is not clear to what extent electric transport will be deployed over the coming years and, moreover, there is the question of whether smart appliances will help to shift battery charging to off-peak times, or whether daily price variations will be too small to influence people's behaviour.

Electricity storage is one potential option for improving the flexibility and reliability of our electricity system.  It could offer services such as peak-shaving, frequency response, reactive power regulation and the provision of reserve.  In the UK, \lf{large-scale} electricity storage currently comes only in the form of pumped hydro power plants, which are implemented largely to meet early evening winter peak demands between 4pm and 8pm \cite{ERP}.  Their main sources of revenue are through providing the balancing services, Short-Term Operating Reserve (STOR) and Fast Reserve, which are funded by National Grid.  The stores then replenish their supply during the night, when electricity prices are at their lowest \cite{Mac}.  %Hence, there is an underlying 
%periodicity in the operation of such a store, which results from the periodicity and predictability of supply, demand and prices.  

%As discussed above, however, our future supply and demand patterns are set to look very different from today and it seems unlikely that an electricity store in a few decades' time will face such strong periodicity of prices.  An important question is whether storage could be a financially viable option in this new setting. 

Much research has been dedicated to analyzing storage viability.  Some papers focus on the benefits which a store could bring to the electricity system \lf{if operated in a socially optimal way} (see, for example, \cite{AEA, Poy, BlaStr, Str}).  Typically, the approach of these papers is to solve the unit dispatch problem i.e. to select a suitable configuration of generators (including storage) to run at each time, in order to ensure that demand is met (or that demand is met with a certain high probability).  An advantage of this method is that it can incorporate the interactions between all assets of the energy system into a single optimization problem, allowing a comparison between the actions of storage against its competing options.  A drawback, however, is that the system optimum does not necessarily coincide with each individual firm's optimal strategy.  In many cases, it is not even clear that each firm would make a profit under these solutions.  Therefore, the whole-system approach is generally better suited to questions 
where the store is not privately owned, but instead owned and operated by a central controller, such as the system operator.

Other papers focus on the profits available to a store which faces stochastic or probabilistic prices.  Some of these papers restrict the behaviour of the store to a pre-defined set of operating strategies (e.g. \cite{BarInf, Gru}).  Another approach is to implement Dynamic Programming techniques (e.g. \cite{Ari, LohMin}).  A drawback of this latter approach, however, is that \short{it tends to be computationally heavy and requires} information about prices over the entire time horizon over which we wish to optimize.  Additionally, such methods do not give much scope for mathematical insight into the dynamics of the solutions.    

In this paper, we present a method to determine the maximum profit available to an electricity store which is operating in the wholesale electricity market.  \short{We assume that} %The main assumption is that 
the store can predict an electricity price function $p:[0,T]\to\re$ \lfff{(or alternatively, a cost function)} where $[0,T]$ is the period of time over which we wish to optimize.  However, the method only ever uses price information over a smaller time interval, thus \lfff{reducing the required prediction horizon and the amount of computation needed.}  \short{It is worth commenting here that the prediction of future prices in this way is} not an unreasonable assumption, since most electricity today is traded through ``over-the-counter'' bilateral contracts, which are agreed ahead of operating time.  

Our method is derived from standard Calculus of Variations techniques and is intended as an extension to \cite{CFGZ}.  Our work differs from \cite{CFGZ} in several respects.  Firstly, we present our model in a continuous time setting, rather than the discrete setting employed in \cite{CFGZ}.  Even if prices are declared at discrete time intervals, in accordance with our current market system, our approach allows for the input of piecewise constant prices but with continuous variation in the operation of the store.  Secondly, \cite{CFGZ} allows only for convex cost functions, whereas here we allow for much more general operating costs and, in particular, we remove the convexity assumption.  This is realistic: in general, there is little reason to believe that the cost of running a motor, for example, is a convex function of the output power.  Moreover, there may be cost jumps involved in switching on a new motor if additional power is required.  Although these more general conditions do not guarantee 
the existence of 
solutions, we prove that optimal operating strategies of the form given in Proposition~\ref{prop:ps} exist if and only if the associated algorithm does not terminate early.  \lff{Finally, we allow the inclusion of additional operating constraints.  If these constraints can be expressed as bounds on the instantaneous power output (e.g. in the case of enforced periods of closure for maintenance), then they can easily be incorporated into the main result (Proposition~\ref{prop:ps}) and the optimal strategy can be determined using the algorithm of Section~\ref{sec:alg}.  This method reduces the optimization problem to a set of smaller problems which are localized in time.  For other constraints, however, this localization in time is not guaranteed.  In Section~\ref{sec:switching}, we demonstrate this issue and propose a method for handling an additional condition if the store must wait for a fixed amount of time (called the minimum switching time) whenever it wishes to change its mode of operation between the charging and the discharging mode.}

The structure of the paper is as follows: Over the remainder of this section, we introduce our storage model and its associated costs. In Section 2, we present the main result, which characterizes the optimal strategies via a reference price function.  In Section 3, we present an algorithm which determines both the optimal strategy and the reference price \lf{for the basic case where the store is only constrained by its power ratings and its capacity constraints}.  \lf{We prove that an optimal strategy exists in this case only if the algorithm does not terminate early. \lff{Section~\ref{sec:constraints} contains a discussion of a simple storage model and in Section~\ref{sec:switching} we investigate the inclusion of a minimum switching time constraint.}  In Section 5, we present and discuss some illustrative results.  Finally, the conclusions of the paper are summarized in Section 6.}

\subsection{The storage model}\label{sec:model}
We use a simple technology-agnostic model for storage, which incorporates a number of key features, both physical and economic.  To allow for multiple applications, we define all quantities in a general setting.  Towards the end of the paper, in Section~\ref{sec:constraints}, we present a concrete (and technically relevant) example.

Let $[0,T]$ be the interval of time over which we want to optimize the actions of the store, for some $T>0$.  \lff{At each time} $t\in[0,T],$ we denote by $\lff{U(t)}\subset\re$ the set of admissible power outputs associated with the store \lff{at that time}.  The operator of the store may then choose an \textit{operating strategy} $q:[0,T]\to\re$ which allocates, at each time $t,$ the amount of power $q(t)\in\lff{U(t)}$ that is to enter the store (using the convention that if $q(t)<0,$ then the store discharges $-q(t)$ \lf{units of power} at time $t$).  \lff{Allowing the power constraint set $U(t)$ to vary with time in this way means that we may incorporate constraints such as planned periods of closure or transmission congestion into the model.  A planned closure over a time period} $I\subset[0,T],$ \lff{for example, corresponds to the power constraint set}
%\begin{align*}
$U(t)=\{0\}$ for all $t\in I.$
%\end{align*}

The amount of stored energy, or the \textit{level of the store}, which results from the strategy $q$ at time $t$ is denoted $\ell[q](t)$ and evolves according to the differential equation
\begin{align}\label{elldefn}
 \frac{d}{dt}\ell[q](t)=-\alpha\ell[q](t)+q(t).
\end{align}
Here, $\alpha\in[0,1)$ is the leakage rate, which is introduced to reflect the self-discharge which occurs due to imperfect sealing or porous casing.  Given an initial condition $\ell[q](0)=\ell_0,$ the level of the store at each time $t\in[0,T]$ is thus given by
\begin{align}
 \ell[q](t)=e^{-\alpha t}\ell_0+e^{-\alpha t}\int_0^te^{\alpha s}q(s)ds.
\end{align}

%\begin{figure}[H]
%\begin{center}
%\includegraphics[scale=0.3, trim=0cm 5cm 0cm 2cm]{Model}
%\end{center}
% \end{figure}
We will %sometimes refer to a power output function $q:[0,T]\to U$ as a ``strategy,'' and will 
only consider strategies $q:[0,T]\to\re$ which are piecewise continuous functions of time.  This is a physically reasonable assumption, since any strategy which has an accumulation point of discontinuities or a non-jump discontinuity would be hard to achieve by any piece of equipment.  Throughout the remainder of this paper, we denote by $\mathcal{C}_{pw}(I)$ the space of piecewise continuous functions $u:I\to\re,$ for any subset $I\subset\re.$    

The \textit{capacity constraints} of the store are characterized by two functions $E^+,E^-:[0,T]\to\re,$ so that any strategy $q:[0,T]\to U$ is constrained by the inequalities
\begin{align*}
 E^-(t)\leq\ell[q](t)\leq E^+(t) \ \ \ \ \ \ \ \ \ \ \forall t\in[0,T].
\end{align*}
We assume that the initial and terminal levels of the store are pre-specified, so that $$E^-(0)=E^+(0) \ \ \ \ \ \textnormal{and} \ \ \ \ \ E^-(T)=E^+(T).$$
Note that, if not specified, then the terminal condition is in fact implicit since any optimal strategy $q^*$ should satisfy that $\ell[q^*](T)=E^-(T).$  \lff{We also assume, without loss of generality, that} $\{(d/dt)E^-(t)+\alpha E^-(t),(d/dt)E^+(t)+\alpha E^+(t)\}\subset U(t)$ \lff{for all} $t\in[0,T]$ (otherwise $E$ could be adjusted so that this is true, without changing the set of admissible strategies).

The set of possible levels of the store, at each time, is represented by the \textit{admissible energy domain} $E\subset[0,T]\times\re,$ which is the set of all pairs $(t,m)\in[0,T]\times[E^-(t),E^+(t)].$  Often, one may choose to replace the capacity constraints $E^+$ and $E^-$ with constants, so that $E^+>0$ is the physical size of the store and $E^-\geq 0$ is the minimum technically feasible level of the store (which may be strictly positive, for example in the case of compressed air energy storage, where a cushion of air needs to be maintained in the store at all times in order to provide sufficient pressure to instigate discharging).  We have chosen here to represent $E$ in this more general form, to allow the store to participate in multiple markets.  For example, if a store decides to participate in a Balancing Mechanism reserve contract, it may be required by National Grid to make supply available over certain agreed periods.  Thus, at these contracted times the store will have less capacity 
available for use in the wholesale market.  

\lff{In addition to the capacity and power constraints outlined above, the operation of a store may be restricted by further factors, such as warm-up times, ramping constraints and minimum switching times between modes of operation.  We denote the collection of additional constraints associated with the store by $\Gamma.$  The only condition on these constraints is that they must be entirely expressible as functions of the power output function $q.$  %  A planned closure over a time period $I\subset[0,T],$ for example, may be encoded by the constraints
%\begin{align*}
%q(t)=0 \qquad \forall t\in I.
%\end{align*}
A maximum ramp-up rate $r>0$ on the charge side, for example, may be written as inequalities 
$$\dot q(t)\leq r \qquad \forall t\in[0,T] \ \textnormal{such that} \ q(t)>0.$$  \lff{We do not allow for constraints of this kind in Section~\ref{sec:nocon}.  However, the impact of a non-empty set $\Gamma$ is discussed in Section~\ref{sec:constraints}.}
%We denote the collection of additional constraints associated with the store by $\Gamma.$  
}
%to denote the set of In general, therefore, we consider the set $\Gamma$ to be a set of equalities of the form $g[q]=0$ and inequalities of the form $h[q]\leq 0,$ for a suitable collection of functionals $g,h:\mathcal C_{pw}([0,T])\to\re$ which encode the operating constraints of the store.}

The total cost incurred by the store, given an operating strategy $q,$ is denoted $C[q]\in\re.$  Here, $C$ is a functional which maps $\Cpw$ into $\re.$  This can incorporate factors such running costs, warming-up costs and costs of storing, as well as the cost of purchasing power and the payments (counted as negative cost) received for providing power.  The aim of this paper is to identify those operating strategies $q$ which minimize the total cost $C[q],$ whilst adhering to all of the physical constraints outlined above.  If the store is profitable to run, then the minimal cost should, of course, be negative.

\begin{defn}[Admissible strategies and optimal strategies]\label{defn:storage}
We say that $q:[0,T]\to\re$ is an admissible strategy if $q$ is piecewise continuous and satisfies the following properties:%\lff{if, for each} $t\in[0,T],$ \lff{we have} $q(t) \in U(t)$ and
\begin{enumerate}
 \item For each $t\in[0,T],$ we have $q(t)\in U(t)$ and
$\ell[q](t)\in[E^-(t),E^+(t)].$  
\item \lf{If $\Gamma\neq\emptyset,$ then the strategy $q$ satisfies each of the constraints in $\Gamma.$}
\end{enumerate}
We denote by $X$ the set of all admissible strategies and say that $q^*\in X$ is an \textit{optimal strategy} if  
  \begin{align}\label{minprob}
  C[q^*]\leq C[q] \ \ \ \ \ \ \ \ \ \ \ \ \ \forall q\in X.
  \end{align}
 
\end{defn}

\section{Characterization of optimal strategies}
The following proposition provides a characterization of optimal strategies in terms of a ``reference price'' function $\mu^*.$  \lf{We introduce the \textit{indicator functional} $S:\mathcal C_{pw}\to \{0,1\},$ which is defined by the property that $S[q]$ takes the value 0 if and only if $q$ satisfies each of the constraints in $\Gamma.$  }

\begin{prop}[Characterization of optimal strategies]\label{prop:ps}
 Suppose there exists a constant \update{$\rho^*\in\re,$} a piecewise differentiable function $\mu^*:[0,T]\to\re$ and a strategy $q^*\in X$ with the following properties: 
 \begin{enumerate}[(i)]
  \item The strategy $q^*$ is a minimizer of
\begin{align}\label{eq:propmax}
C[q]-\int_0^Te^{\alpha t}\mu^*(t) q(t)dt+\rho^*S[q]
\end{align}
over all piecewise continuous functions $q:[0,T]\to\re$ \lff{such that} $q(t)\in U(t)$ \lff{for all} $t\in[0,T].$
\item If $\mu^*$ is differentiable at $t\in[0,T],$ then the following complementary slackness conditions are satisfied:
\begin{align}
\dt{\mu^*}(t)=0 \ & \ \ \ \textnormal{if} \ E^-(t)<\ell[q^*](t)<E^+(t),\label{eq:mucond1}\\
\dt{\mu^*}(t)\geq 0 \ & \ \ \ \textnormal{if} \ \ell[q]^*(t)=E^+(t),\label{eq:mucond2}\\
\dt{\mu^*}(t)\leq 0 \ & \ \ \ \textnormal{if} \ \ell[q^*](t)=E^-(t).\label{eq:mucond3}
\end{align}
\item If $\mu^*$ is not differentiable at $t\in[0,T],$ then the following ``jump'' complementary slackness conditions are satisfied:
\begin{align}
 \mu^*(t^+)-\mu^*(t^-)=0 \ & \ \ \ \textnormal{if} \ E^-(t)<\ell[q^*](t)<E^+(t),\label{eq:mucond1'}\\
\mu^*(t^+)-\mu^*(t^-)\geq 0 \ & \ \ \ \textnormal{if} \ \ell[q^*](t)=E^+(t),\label{eq:mucond2'}\\
\mu^*(t^+)-\mu^*(t^-)\leq 0 \ & \ \ \ \textnormal{if} \ \ell[q^*](t)=E^-(t)\label{eq:mucond3'},
\end{align}
where $\mu^*(t^-)$ and $\mu^*(t^+)$ are the left and right limits respectively of $\mu^*$ at $t.$
 \end{enumerate}

Then $q^*$ is an optimal strategy.
\end{prop}

\begin{proof}
Let $q\in X$ and let $0=a_1<\ldots<a_n<a_{n+1}=T$ be a partition such that $\cup_{i=1}^n[a_i,a_{i+1}]=[0,T]$ and $q^*,q$ are continuous and $\mu^*$ is differentiable over each $(a_i,a_{i+1}).$  Notice that (\ref{elldefn}) can be equivalently written as
$$\frac{d}{dt}e^{\alpha t}\ell[q](t)=e^{\alpha t}q(t) \ \ \ \ \ \ \ \ \ \ \ \forall t\in[0,T].$$
Hence, (\ref{eq:propmax}) implies
\begin{align*}
C[q^*]&-C[q]\leq\int_0^Te^{\alpha t}\mu^*(t)\Big(q^*(t)-q(t)\Big)dt+\rho^*\Big(\underbrace{S[q^*]-S[q]}_{=0}\Big)\\
&\stackrel{(\ref{elldefn})}{=}\sum_{i=1}^n\int_{a_i}^{a_{i+1}}\mu^*(t)\left\{\frac{d}{dt}\Big(e^{\alpha t}\ell[q^*](t)-e^{\alpha t}\ell[q](t)\Big)\right\}dt\\
&=\sum_{i=1}^n\Big[e^{\alpha a_{i+1}}\mu^*(a_{i+1}^-)\left(\ell[q^*](a_{i+1})-\ell[q](a_{i+1})\right)-e^{\alpha a_i}\mu^*(a_i^+)\left(\ell[q^*](a_i)-\ell[q](a_i)\right)\Big]\\
& \ \ \ \ \ \ \ \ \ \ \ \ \ \ \ \ \ \ \ \ \ \ \ \ \ \ \ \ \ \ \ \ \ \ \  \ \ \ \ \ \ \ \ \ \ \ \  -\sum_{i=1}^n\int_{a_i}^{a_{i+1}}e^{\alpha t}\dt{\mu^*}(t)\left(\ell[q^*](t)-\ell[q](t)\right)dt\\
&=\sum_{i=2}^ne^{\alpha a_i}\left(\mu^*(a_i^-)-\mu^*(a_i^+)\right)\left(\ell[q^*](a_i)-\ell[q](a_i)\right)\\
& \ \ \ \ \ \ \ \ \ \ \ \ \ \ \ \ \ \ \ \ \ \ \ \ \ \ \ \ \ \ \ \ \ \ \  \ \ \ \ \ \ \ \ \ \ \ \  -\sum_{i=1}^n\int_{a_i}^{b_i}e^{\alpha (t)}\dt{\mu^*}(t)\left(\ell[q^*](t)-\ell[q](t)\right)dt
\end{align*}
where we have used the fact that, since $q$ and $q^*$ are admissible, then $S[q^*]=S[q]=0.$  The second equality follows from integration by parts and final equality results from a rearrangement of the first sum on the previous line, together with the condition that $\ell[q](0)=\ell[q^*](0)=E^-(0)$ and $\ell[q](T)=\ell[q^*](T)=E^-(T)$ if $q\in X.$  Applying the complementary slackness conditions (\ref{eq:mucond1})-(\ref{eq:mucond3'}) to the final equality above, we obtain $C[q^*]-C[q]\leq 0.$
\end{proof}

\short{The reference price} $\mu^*$ \short{provides a reference cost per unit of power.  Roughly, if the cost per unit of power for operating the store is lower than the reference price, then this indicates the store should charge at that time (and similarly for discharging). }

It is worth mentioning here that exactly the same result follows if we apply the Karush-Kuhn-Tucker (KKT) conditions to \lf{the capacity constraints of the} original minimization problem~(\ref{minprob}).  By following such an approach, one instead looks for minimizers of the relaxed functional
\begin{align*}
 C[q]-\int_0^T\Big(\lambda_1(t)\left(\ell[q](t)-E^-(t)\right)+\lambda_2(t)\left(E^+(t)-\ell[q](t)\right)\Big)dt+\rho^*S[q]
\end{align*}
over all piecewise-continuous $q:[0,T]\to\re$ \lff{such that} $q(t)\in U(t)$ \lff{for all} $t\in[0,T],$ where $\lambda_1,\lambda_2:[0,T]\to\re$ are piecewise-continuous functions which satisfy the complementary slackness conditions 
$$\lambda_1(t)\left(\ell[q](t)-E^-(t)\right)=\lambda_2(t)\left(E^+(t)-\ell[q](t)\right)=0 \ \ \ \ \ \ \ \forall t\in[0,T].$$
The relation between the two approaches is that the reference price function $\mu^*$ in Proposition~\ref{prop:ps} can be expressed as an integral of the KKT functions:
$$\mu^*(t)=\int_t^T\left(\lambda_1(s)-\lambda_2(s)\right)ds$$ and, in particular, $(d\mu^*/dt)(t)=-\lambda_1(t)+\lambda_2(t).$

\lff{The following lemma shows that, in certain cases, we should expect the determination of minimizers $q^*$ to reduce to a sequence of smaller minimization problems which are localized in time.}

%We now allow for a general set of additional constraints $\Gamma.$  \short{We assume again that the cost functional} $C$ \short{takes the form} (\ref{C}) and, for simplicity, we further assume that $E^-(t)=0$ and $E^+(t)=M>0$ for all $t\in(0,T),$ and $E^-(0)=E^+(0)=E^-(T)=E^+(T)=0.$   \lf{Recall from Section~\ref{sec:model} that the set $\Gamma$ consists of the physical constraints which restrict the operation of the store, in addition to the usual power and capacity constraints.  Each of these constraints must be expressible entirely in terms of functions of the power output of the store.}  %
%The previous algorithm is no longer sufficient under these more general conditions.  A common issue is that the value that $\mu^*$ attains over regions where it is constant may now depend on subsequent values of $\mu^*.$  \lf{This issue is discussed in more detail in Section~\ref{sec:switching}, where we consider the specific example of minimum switching time constraints.}  The following lemma allows us to circumvent this problem to an extent, \short{stating} that optimal strategies can be determined by looking \lf{only at how the cost function varies locally in time.}

\begin{lemma}\label{lemma:constraints}[Localization of the minimization problem]
\lff{Let $E:[0,T]\to\re$ be the admissible domain, $U(t)\subset\re$ be the set of admissible power outputs at each time} $t\in[0,T]$ \lff{and let $C$} be a cost functional \short{of the form}
\begin{align}\label{C}
 C[q]=\int_0^TL(t,q(t)) \ dt
 \end{align}
\lff{for some running cost function} $L:\mathcal A\to\re,$ \lff{where} $\mathcal A$ \lff{is the set of time--power pairs} $(t,x)\in[0,T]\times\re$ \lff{such that} $x\in U(t).$ \lff{For each} $t\in[0,T],$ \lff{let} $\ell_1(t)=E^{-}(t)$ and $\ell_{2}(t)=E^+(t).$  Assume that there exists an optimal strategy which satisfies the conditions of Proposition~\ref{prop:ps} and \short{suppose that the optimal level of the store at some time} $\sigma_0\in[0,T]$ \short{is known to be} $\ell=\ell_i(\sigma_0),$ with $i\in\{1,2\}.$  For any $\sigma\in(\sigma_0,T],$ let $q_{\sigma}$ solve the minimization problem 
\begin{align}\label{eq:lemmamin}
\int_{\sigma_0}^{\sigma}L(t,q_{\sigma}(t)) \ dt\leq \int_{\sigma_0}^{\sigma}L(t,q(t)) \ dt
\end{align}
over all piecewise continuous $q:[\sigma_0,\sigma]\to\re$ such that $q(t)\in U(t)$ \lff{for all} $t\in[0,T]$ and \lff{such that the end conditions} $\ell[q](\sigma)=\ell_i(\sigma)$ and $\ell[q](\sigma_0)=\ell$ hold.  \short{If there exists a time} $t_0\in(\sigma_0,\sigma)$ \short{such that} $\ell[q_{\sigma}](t_0)=\ell_j(t_0),$ \lff{with} $j\in\{1,2\}\setminus\{i\},$ then there exists an optimal strategy $q^*\in X$ \short{which coincides exactly with} $q_{\sigma}$ \short{over the interval} $(\sigma_0,t_0).$
 \end{lemma}
 \begin{proof}
 Assume first that $\ell=\ell_1(\sigma_0)$ so that $\ell[q_{\sigma}](t_0)=\ell_2(t_0)$ and $\ell[q_{\sigma}](\sigma)=\ell_1(\sigma).$  Let $q^*\in X$ be an optimal strategy, let $q_{\sigma}$ satisfy the hypotheses above and suppose that the claim of the lemma is not true.  Then $\ell[q^*](\sigma)>\ell_1(\sigma)$ (since otherwise we could replace $q^*$ with $q_{\sigma}$ over $[0,\sigma]$ to either reduce the total cost \lff{or keep the cost the same}).  Thus there must exist at least one interval $(a,b)\subset(\sigma_0,\sigma),$ with $b<\sigma,$ such that $\ell[q^*](a)-\ell[q_{\sigma}](a)=\ell[q^*](b)-\ell[q_{\sigma}](b)=0$ and $\ell[q^*](t)\neq\ell[q_{\sigma}](t)$ for all $t\in (a,b).$  However, this implies that both $q^*$ and $q_{\sigma}$ must be minimizers of $\int_a^bL(t,q(t)) \ dt$ over all \lff{admissible} $q:[a,b]\to U$ such that $\ell[q](a)=\ell[q^*](a)$ and $\ell[q](b)=\ell[q^*](b).$  In particular, we may adapt $q^*$ if necessary so that $q^*(t)=q_{\sigma}(t)$ for all $t\in(a,b).$  Repeating this argument for all such intervals $(a,b)$ completes the proof that the optimal strategy can be chosen to coincide with $q_{\sigma}$ over $(\sigma_0,t_0).$  A similar argument holds for the case where $\ell=\ell_2(\sigma_0).$

  \end{proof}

\short{The significance of the above lemma is that $q_{\sigma}$ depends only on knowledge of a previous time $\sigma_0$ when the store is either full or empty, and on the action of the cost functional over the time interval $[\sigma_0,\sigma].$  In particular, starting at time 0, one sets $\ell=E^-(0)=E^+(0)=0$ and finds a time $\sigma\in(0,T]$ such that the minimizer $q_{\sigma}$ of (\ref{eq:lemmamin}) hits the upper capacity constraint at some time in $t_0\in(0,\sigma).$  The optimal strategy $q^*\in X$ can thus be chosen to coincide with $q_{\sigma}$ over the interval $[0,t_0].$  We then set $\ell=E^+(t_0)$ and $\sigma_0=t_0$ and repeat the process.  Continuing in this way, we eventually construct a global optimizer $q^*.$}

\section{Determination of optimal strategies with no additional constraints}\label{sec:nocon}
We assume in this section that the only constraints faced by the store are its capacity and power constraints (i.e. $\Gamma=\emptyset$).  \lff{For simplicity of notation, we assume from now on that there is a fixed set of power constraints $U\subset\re$ such that $U(t)=U$ for all} $t\in[0,T].$  \lff{All of the results extend naturally to the more general case where} $U$ \lfff{is allowed to vary with time.}  We present an algorithm to determine $\mu^*,$ and consequently the optimal strategy $q^*,$ for the class of functionals $C$ which take the form~(\ref{C}) 
%\begin{align}\label{C}
%C[q]=\int_0^TL(t,q(t))dt 
%\end{align}
for some $L:[0,T]\times U\to\re$ with $L(t,0)=0$ for all $t\in[0,T].$  We assume that, for each $t\in[0,T],$ the map $\lambda\mapsto L(t,\lambda)$ is piecewise differentiable and that the associated partial derivative $\partial L(t,\lambda)/\partial\lambda$ has a continuous inverse at almost every $\lambda\in U.$  \lf{This class of functionals contains costs which depend only on the time $t$ and the power output of the store at that time.  The piecewise differentiability assumption allows for jumps in the cost.  These jumps may relate to the cost of switching on an additional motor in order to provide a higher power output, for example.}

Proposition~\ref{prop:ps} states that, if we can find the appropriate reference price function $\mu^*,$ then the optimal strategy $q^*\in X$ solves 
\begin{align}\label{argmin}
\L[q^*,\mu^*]\leq\L[q,\mu^*] \ \ \ \ \ \ \forall \ q:[0,T]\to U,
\end{align}
where for any functions $q,\mu:[0,T]\to\re,$ we define the relaxed functional
\begin{align}\label{relaxed}
 \L[q,\mu]=\int_0^T\Big(L(t,q(t))-\lf{e^{\alpha t}}\mu(t)q(t)\Big) \ dt.
\end{align}
The regularity assumptions on $L,$ \ $q$ and $\mu$ imply that (\ref{argmin}) reduces to a minimization problem which is pointwise in time.  Specifically, for almost all $t\in[0,T],$ we search for $q^*(t)\in U$ which solves
\begin{align}\label{minptwise}
 L(t,q^*(t))-\lf{e^{\alpha t}}\mu^*(t)q^*(t)\leq L(t,w)-\lf{e^{\alpha t}}\mu^*(t)w \qquad \forall w\in U.
\end{align}

Our algorithm is a generalization of that provided in \cite{CFGZ}.  An important result is Proposition~\ref{prop:algpot2}, which states that if the algorithm does not terminate early, then it does indeed provide an optimal strategy $q^*.$  In particular, in the special case that the cost functional $C$ is strictly convex, then there exists a unique optimal strategy, and this coincides exactly with $q^*.$

\subsection{Preliminary results and definitions}
By Proposition~\ref{prop:ps}, we aim to find a piecewise differentiable $\mu^*:[0,T]\to\re$ such that the optimal strategy $q^*\in X$ solves (\ref{argmin}).  Moreover, $\mu^*$ should be constant over intervals of time where the level of the store is away from the capacity constraints.  This motivates the following construction:  

Given $t\in[0,T]$ and $\lambda\in\re,$ we want to define a quantity $u(t,\lambda)\in\re$ as a solution of
\begin{align}\label{algmin1}
u(t,\lambda)=\arg \min_{w\in U}\Big(L(t,w)-e^{\alpha t}\lambda w\Big).
\end{align}
If we know that $\mu\equiv\lambda$ over some connected interval $I\subset[0,T],$ then we hope that $q^*(t)=u(t,\lambda)$ for all $t\in I.$   Care needs to be taken, however, because there may be multiple minimizers associated with  $\lambda.$  With this in mind, we denote by $\mathcal{M}_t\subset\re$ the set of $\lambda$ which admit multiple minimizers in (\ref{algmin1}).  Note that, if $L$ is assumed to be strictly convex in its second argument, then minimizers of (\ref{eq:propmax}) are of course unique, implying that each $\mathcal{M}_t=\emptyset.$

\begin{lemma}\label{lemma:xinc}
For each $(t,\lambda)\in[0,T]\times\re,$ let $u(t,\lambda)$ be any solution of (\ref{algmin1}).  Then, at each $t\in[0,T],$ the mapping $\lambda\mapsto u(t,\lambda)$ is monotone increasing and piecewise continuous.  Moreover, the set $\mathcal{M}_t$ is finite and discontinuities in $\lambda\to u(t,\lambda)$ occur only at points in $\mathcal{M}_t.$ 
\end{lemma}
\begin{proof}
Let $t\in[0,T]$ and suppose that the map $\lambda\mapsto u(t,\lambda)$ is not monotone increasing.  Then, there exists $\lambda_1<\lambda_2$ such that $u_1:=u(t,\lambda_1)>u(t,\lambda_2)=:u_2.$  But then,
 \begin{align*}
  L(t,u_1)-e^{\alpha t}\lambda_2u_1%&=L(t,u_1)-e^{\alpha t}\lambda_1u_1-e^{\alpha t} (\lambda_2-\lambda_1)u_1\\
  &\leq L(t,u_2)-e^{\alpha t}\lambda_1u_2-e^{\alpha t}(\lambda_2-\lambda_1)u_1
%  &=L(t,u_2)-e^{\alpha t}\lambda_2u_2-e^{\alpha t}(\lambda_2-\lambda_1)(u_1-u_2)\\
  <L(t,u_2)-e^{\alpha t}\lambda_2u_2.
 \end{align*}
The first inequality follows from the definition of $u_1$ and $u_2$ as solutions to (\ref{algmin1}), and the second inequality follows from the supposition.  However, the above contradicts the definition of $u_2,$ and we conclude that the map $\lambda\mapsto  u(t,\lambda)$ is monotone increasing at all $t\in[\tau,T].$

The piecewise continuity of the map $\lambda\mapsto u(t,\lambda)$ follows immediately from the regularity assumptions on $L.$  Precisely, if for almost all $\lambda\in\re,$ the minimizers $u(t,\lambda)$ lie away from discontinuities of $L,$ then they must satisfy that
\begin{align*}
\frac{\partial L}{\partial\lambda}(t,\lambda)-e^{\alpha t}\lambda=0 \ \ \ \ \ \textnormal{a.e.} \ \lambda\in\re.
\end{align*}
The continuous invertibility assumption on the partial derivative of $L$ therefore implies the piecewise continuity of the map $\lambda\mapsto u(t,\lambda).$  If, on the other hand, there is a non-degenerate subset $W\subset\re$ such that if $\lambda\in W$ then $u(t,\lambda)$ lies at a discontinuity of $L,$ then the above monotonicity property and the piecewise continuity of $L$ imply that the map $\lambda\mapsto u(t,\lambda)$ must be piecewise constant over $W.$  In either case, the map $\lambda\mapsto u(t,\lambda)$ is piecewise continuous.

Finally, we prove the finiteness of the set $\mathcal M_t$ by supposing that the opposite is true.  To this end, let $\mathcal{A}\subset\re$ be an infinite set such that, for each $\lambda\in\mathcal{A},$ there are two distinct local solutions $x(\lambda),y(\lambda)\in U\setminus\partial U$ to (\ref{algmin1}), where $\partial U$ denotes the boundary of the set $U.$  Let $\lambda_0\in\mathcal{A}$ be such that there exists a sequence $(\lambda_n)_{n\in\n}$ in $\re$ with $\lim_{n\to\infty}\lambda_n=\lambda_0,$ and set $x(\lambda_0)=x_0$ and $y(\lambda_0)=y_0.$  Assume without loss of generality that
\begin{align*}
L_u(t,x(\lambda))=L_u(t,y(\lambda))=e^{\alpha t}\lambda
\end{align*}
is satisfied at $\lambda=\lambda_0,$ where $L_u(t,\cdot)$ denotes the partial derivative of $L$ with respect to the second argument, and that $x(\lambda)$ and $y(\lambda)$ lie away from any discontinuity of $L.$  Setting $g(\lambda):=L(t,x(\lambda))-\lambda x(\lambda)$ and $h(\lambda):=L(t,y(\lambda))-\lambda y(\lambda),$ we obtain
\begin{align*}
 g'(\lambda_0)&=L_u(t,x_0)x'(\lambda_0)-x'(\lambda_0)\lambda_0-x_0=\left(L_u(t,x_0)-\lambda_0\right)x'(\lambda_0)-x_0=-x_0
\end{align*}
and similarly, $h'(\lambda_0)=-y_0.$  Hence, if $x_0$ and $y_0$ are both global minimizers solving (\ref{algmin1}) but $x_0\neq y_0,$ then there exists $N\in\n$ such that $g(\lambda_n)\neq h(\lambda_n)$ for all $n\geq N.$  In particular, $x(\lambda_n)$ and $y(\lambda_n)$ cannot both be minimizers for $n\geq N.$  This contradicts the assumption that there exists such a set $\mathcal{A}.$
\end{proof}

The definitions which follow will be employed in the algorithm, and are written under the assumption that there exists a pair $(\mu^*,q^*)$ which satisfies the complementary slackness conditions of Proposition~\ref{prop:ps}.  Condition~(\ref{eq:mucond1}) motivates the following definition.

\begin{defn}[$(\lambda,\tau)$-admissibility]
For any choice of $\lambda\in\re$ and $\tau\in[0,T],$ we say a piecewise continuous function $x:[0,T]\rightarrow U$ is $(\lambda,\tau)-$admissible if $x\equiv 0$ over $[0,\tau]$ and if, for each $t\in[0,T],$ $x(t)$ is a solution of
\begin{align}\label{algmin}
x(t)=\textnormal{arg}\min_{w\in U}\Big(L(t,w)- e^{\alpha t}\lambda w\Big).
\end{align}
In particular, the mapping $t\mapsto x(t)$ coincides with $t\mapsto u(t,\lambda)$ over $[\tau,T],$ for some choice of map $u$ which solves (\ref{algmin1}) at each $t\in[0,T].$
\end{defn}

Note that, in general, a $(\lambda,\tau)-$admissible strategy $x$ is not admissible: for a general choice of $\lambda,$ one of the capacity constraints is likely to be broken at some time in $(\tau,T].$  The idea of the algorithm is to piece together $(\lambda,\tau)-$admissible strategies in order to construct the optimal strategy $q^*\in X.$  \short{The reference price} $\mu^*$ \short{attains the appropriate value of} $\lambda$ \short{over intervals of time when} $\mu^*$ \short{is constant}  %The construction will require the choice of $\lambda$ to be updated at a discrete subset of times in $[0,T].$  \lf{The reference price function $\mu^*$ of Proposition~\ref{prop:ps} then attains the appropriate value of $\lambda$ over intervals of time when $\mu^*$ is constant.}  
In the special case that one can find a $(\lambda,0)-$admissible strategy which is also admissible, then of course we are done, since we may set $\mu^*\equiv\lambda.$  Such a strategy satisfies condition~(\ref{eq:mucond1}) at every $t\in[0,T].$

The key feature of the algorithm is thus to determine: (i) the intervals of time over which $\mu^*$ is constant and (ii) the value that $\mu^*$ attains over these intervals.  \short{Intuitively, we expect the store to completely fill whenever costs are sufficiently low.  The store will then wait for costs to rise before fully discharging.  Correspondingly, the reference price should increase over this waiting period, from} $\lambda_1$ say to $\lambda_2$.  \short{If we did not allow for this increase in reference price, then the associated strategy would cause the store to begin discharging too soon and to empty too much, thus breaking the lower capacity constraint in order to take advantage of the low costs.  This behaviour is an indication that} $\lambda_1$ \short{is the correct choice of reference price over the charging period.  This motivates the following characterization of strategies.}
%The following definition provides a characterization of strategies which assists with this.  However, it is useful first to provide some intuition.  To this end, suppose that the store is empty at a time $\tau$ and that we choose a high reference price $\mu^*(\tau)$ at that time.  Then, if $\mu^*(\tau)$ is high enough, the store begins charging up at time $\tau$ and hence $\mu^*$ must remain fixed until a later time $\tau'$ at which the store is either full or empty.  If $\mu^*(\tau)$ is sufficiently high, then we would find that the first possible $\tau'$ is a time at which the store is full, rather than empty, and that if we did not adjust $\mu^*$ at time $\tau',$ then the store would continue charging past this point, thus breaking the upper capacity constraint.  The monotonicity property of Lemma~\ref{lemma:xinc} implies that, in this case, our only option would be to reduce the value of $\mu^*$ at time $\tau',$ which means that the reference price no longer satisfies conditions~\ref{eq:mucond2} or (\ref{eq:mucond2'}).  Our very high choice of $\mu^*(\tau)$ was therefore too high.  If the reference price $\mu^*(\tau)$ needs to be updated at a time $\tau'$ at which the store is full, then $\mu^*(\tau)$ must be just high enough so that the associated strategy reaches the upper constraint (but does not exceeed it) before discharging back to the empty state (and possibly breaking the lower capacity constraint).  We may then increase $\mu^*$ when the store is full, and this will lift the level of the store at subsequent times.

\begin{defn}[Characterization of $(\lambda,\tau)$-admissible strategies]\label{fnset}
Let $\tau\in[0,T],$ \ $\lambda\in\re$ and $m\in[0,M].$  We characterize each $(\lambda,\tau)-$admissible $x:[0,T]\to U$ according to sets $X(\tau,m)$ and $X'(\tau,m)$ as follows:     
\begin{enumerate}[(i)]
 \item We write $x\in X(\tau,m)$ if there exists $t\in(\tau,T)$ and $\varepsilon_0>0$ such that 
  \begin{align*}
m+\ell[x](t)&=E^-(t) \qquad \textnormal{and} \qquad  m+\ell[x](t+\varepsilon)<E^-(t+\varepsilon)
\end{align*}
for all $\varepsilon\in(0,\varepsilon_0),$ with 
\begin{align}\label{defncond}
E^-(t')\leq m+\ell[x](t')\leq E^+(t') \qquad \forall t'\in(\tau,t). 
\end{align}

In other words, if the level of the store at time $\tau$ is $m,$ and if the store adopts the strategy $x$ over $(\tau,T],$ then the first violation of a capacity constraint occurs at the lower bound of $E.$  We write $t=t[x,m].$ 
\item Similarly, we write $x\in X'(\tau,m)$ if there exists $t\in(\tau,T)$ and $\varepsilon_0>0$ such that condition~(\ref{defncond}) holds and 
  \begin{align*}
m+\ell[x](t)&=E^+(t) \qquad \textnormal{and} \qquad  m+\ell[x](t+\varepsilon)>E^+(t+\varepsilon)
\end{align*}
for all $\varepsilon\in(0,\varepsilon_0).$  In other words, if the level of the store at time $\tau$ is $m,$ and if the store adopts the strategy $x$ over $(\tau,T],$ then the first violation of a capacity constraint occurs at the upper bound of $E.$  We write $t=t[x,m].$
\item We say $\lambda\in\Lambda(\tau,m)$ if there exists a $(\lambda,\tau)-$admissible $x\in X(\tau,m).$  Similarly, $\lambda\in\Lambda'(\tau,m)$ if there exists a $(\lambda,\tau)-$admissible $x\in X'(\tau,m).$

%\item If $x\in X(\tau,m)\cup X'(\tau,m),$ then we denote by $s[x](\tau,m)\in(\tau,T]$ the first time at which a capacity constraint is broken.

\end{enumerate}

\end{defn}
%\lf{Finally, the following definition splits the set of possible reference price values $\lambda$ into two sets: $\Lambda(\tau,m)$ and $\Lambda'(\tau,m).$    
%These two sets allow us to characterize each $\lambda$ according to which constraint is broken first by the associated $(\lambda,\tau)-$admissible strategy.}

%\begin{defn}\label{lambdaset}
% For any $\tau\in[0,T]$ and $m\in[0,M],$ we define
% \begin{align*}
% \Lambda(\tau,m):=\{\lambda\in\re:x\in X(\tau,m) \ \textnormal{for some} \ (\lambda,\tau)-\textnormal{admissible} \ x\}
%\end{align*}
%and
%\begin{align*}
% \Lambda'(\tau,m):=\{\lambda\in\re:x\in X'(\tau,m) \ \textnormal{for some} \ (\lambda,\tau)-\textnormal{admissible} \ x\}.
%\end{align*}
%\end{defn}
Lemma~\ref{lemma:xinc} implies that $\Lambda(\tau,m)$ and $\Lambda'(\tau,m)$ are each connected subintervals of $\re$ which satisfy
\begin{align}\label{infsup}
\sup\Lambda(\tau,m)\leq\inf\Lambda'(\tau,m) \ \ \ \ \ \ \ \ \ \ \ \ \ \forall (\tau,m)\in[0,T]\times[0,M].  
\end{align}
In particular, the interiors of the two sets are disjoint:
\begin{align}\label{eq:int}
 \textnormal{int}(\Lambda(\tau,m))\cap\textnormal{int}(\Lambda'(\tau,m))=\emptyset  \ \ \ \ \ \ \ \ \ \ \ \ \ \forall (\tau,m)\in[0,T]\times[0,M].  
\end{align}

\subsection{The algorithm}\label{sec:alg}

We are now in a position to outline the algorithm.  To this end, fix $\tau\in[0,T)$ and suppose that $\mu^*$ and $q^*$ are known over $[0,\tau].$  We assume without loss of generality that $m:=\ell[q^*](\tau)\in\{E^-(\tau),E^+(\tau)\}$ (otherwise, move $\tau$ backwards until this is satisfied). The algorithm will identify an $N\in\n$ and two increasing sequences of times $\{\tau_i\}_{i=1}^N$,$\{\sigma_i\}_{i=1}^N$ such that $\tau_i\leq\sigma_i$ for each $i,$ and such that $(d/dt)\mu^*(t)=0$ whenever $t\in(\tau_i,\sigma_i).$  The reference price $\mu^*$ is only allowed to jump in value at times of the form $\tau_i$ or $\sigma_i.$  We may assume without loss of generality that $\tau=\tau_k,$ for some $k\in\{1,\ldots,N\}$ (otherwise, again, move $\tau$ backwards until this is true).  \short{Recall from Definition~\ref{fnset} that} $t[x,m]$ \short{is the time at which a strategy $x$ first breaks a capacity constraint, given that} $x(0)=m.$\\
\\
\textbf{Step 1:} If there exists $\lambda\in\re$ and a $(\lambda,\tau_k)-$admissible $x$ such that $x\notin X(\tau_k,m)\cup X'(\tau_k,m),$ then set \lf{$\sigma_k=T.$  Set} $\mu^*(t)=\lambda$ for all $t\in(\tau,T]$ and define the restriction of $q^*$ to $(\tau_k,T]$ to be $x.$  \textbf{The algorithm is complete.}

If there is no $\lambda$ satisfying these conditions, proceed to step 2.\\
 \textbf{Step 2:}  Set $\lambda:=\sup\Lambda(\tau_k,m).$ 
 There are three cases.  For each case, we define the pair $(\mu^*,q^*)$ restricted to an interval $(\tau_k,\sigma_k]\subset(\tau_k,T].$ %$s[x]:=s[x](\tau_k,m).$  

 \begin{itemize}
 \item[a)] If $\lambda\in\Lambda(\tau,m)\setminus\Lambda'(\tau,m):$ Choose a $(\lambda,\tau_k)-$admissible $x$ such that $m+\ell[x](\sigma_k)=E^+(\sigma_k)$ for some  $\sigma_k\in[\tau_k,t[x,m]),$ and select the latest such $\sigma_k.$  Set $\mu^*(t)=\lambda$ for all $t\in(\tau_k,\sigma_k]$ and define the restriction of $q^*$ to $(\tau_k,\sigma_k]$ to coincide with $x.$  \lf{If $\sigma_k<T,$ proceed to step 3; if $\sigma_k=T,$ proceed to step 4.}

 %If no such $x$ exists, then the algorithm terminates here and yields no solution $q.$
 
 \item[b)] If $\lambda\in\Lambda'(\tau,m)\setminus\Lambda(\tau,m):$  Choose a $(\lambda,\tau_k)-$admissible $x$ such that $m+\ell[x](\sigma_k)=E^-(\sigma_k)$ for some  $\sigma_k\in[\tau_k,t[x,m]),$ and select the latest such $\sigma_k.$  Set $\mu^*(t)=\lambda$ for all $t\in(\tau_k,\sigma_k]$ and define the restriction of $q^*$ to $(\tau_k,\sigma_k]$ to coincide with $x.$  \lf{If $\sigma_k<T,$ proceed to step 3; if $\sigma_k=T,$ proceed to step 4.}
 
%If no such $x$ exists, then the algorithm terminates here and yields no solution $q.$

 \item [c)] If $\lambda\in\Lambda(\tau,m)\cap\Lambda'(\tau,m):$  Choose a $(\lambda,\tau_k)-$admissible $x$ such that either $m+\ell[x](\sigma_k)=E^+(\sigma_k)$ or $m+\ell[x](\sigma_k)=E^-(\sigma_k)$ for some  $\sigma_k\in[\tau_k,t[x,m]),$ and select the latest such $\sigma_k.$  Set $\mu^*(t)=\lambda$ for all $t\in(\tau_k,\sigma_k]$ and define the restriction of $q^*$ to $(\tau_k,\sigma_k]$ to coincide with $x.$  \lf{If $\sigma_k<T,$ proceed to step 3; if $\sigma_k=T,$ proceed to step 4.}

 %If no such $x$ exists, then the algorithm terminates here and yields no solution $q.$
 
 \end{itemize}
 \textbf{Step 3:}  \lf{Let $\tau_{k+1}$ be the first time in $[\sigma_k,T]$ such that, on relabeling $\tau_{k+1}$ as $\tau_k$ and $m+\ell[x](\tau_{k+1})$ as $m,$ and on determining the new associated $\sigma_k$ from steps 1 and 2, we have $\sigma_k\neq\tau_k.$}  Relabel $\tau_{k+1}$ as $\tau_k$ and return to step 1.
 
 If there is no such $\tau_{k+1},$ then we have found all intervals $(\tau_i,\sigma_i)$ over which $(d/dt)\mu^*=0.$  Proceed to step 4.\\
 \textbf{Step 4:}  Over each interval $(\sigma_i,\tau_{i+1}],$ corresponding to the times found in the steps above, and for each $t\in(\sigma_i,\tau_{i+1}],$ define
 \begin{align*}
  q^*(t)=\left\{\begin{array}{lll}
                 \frac{d}{dt} E^+(t)+\alpha E^+(t) & & \textnormal{if} \ \ell[q](\sigma_i)=E^+(\sigma_i)\\
                 \frac{d}{dt} E^-(t)+\alpha E^-(t) & & \textnormal{if} \ \ell[q](\sigma_i)=E^-(\sigma_i).
                 \end{array}
\right.
  \end{align*}
  \textbf{The algorithm is complete} if, for each each $i\in\{1,\ldots,N\},$ there exists a piecewise differentiable $\mu_i:(\sigma_i,\tau_{i+1})\to\re$ such that the following conditions hold:
 \begin{itemize}
 \item $q^*(t)$ is a minimizer of $L(t,x)-\mu_i(t)x$ over all $x\in\re.$
 \item If $\ell[q^*](\sigma_i)=E^+(\sigma_i),$ then $\mu_i$ is non-decreasing over $(\sigma_i,\tau_{i+1})$ and satisfies that $\mu_i(\sigma_i^+)\geq\mu^*(\sigma_i^-)$ and $\mu_i(\tau_{i+1}^-)\leq\mu^*(\tau_{i+1}^+).$ 
 \item If $\ell[q^*](\sigma_i)=E^-(\sigma_i),$ then $\mu^*$ is non-increasing over $(\sigma_i,\tau_{i+1}]$ and satisfies that $\mu_i(\sigma_i^+)\leq\mu^*(\sigma_i^-)$ and $\mu_i(\tau_{i+1}^-)\geq\mu^*(\tau_{i+1}^+).$
 \end{itemize}
If these conditions do not hold, the algorithm ends here and yields no solution $q^*.$\\
 \textbf{End of algorithm}\\
%\begin{remark}\label{remark}
The above algorithm implicitly reduces the original optimization problem to a series of new optimization problems which are localized in time.  At each time $\tau_k,$ one only needs to look ahead to a time $t_k$ to know how to operate the store over $(\tau_k,\sigma_k),$ where $t_k$ is defined as $t[x,m]$ for the appropriate choice of $x$ and $m$ from step 2 of the algorithm (or as the the minimum value of $t[x,m]$ whenever there there is a choice of $x$).  Hence, whilst we originally assumed knowledge of the entire price function $p:[0,T]\to[0,\infty),$ in practice we may not need so much information.
%\end{remark}

The following result states that, provided the algorithm does not terminate early, then it does indeed give an optimal strategy $q^*.$  Proposition~\ref{prop:algpot2} then states the converse: if there is an optimal strategy, which satisfies the conditions of Proposition~\ref{prop:ps}, then the algorithm will find it (in other words, the algorithm will not terminate early).
\begin{prop}\label{prop:algopt}
If the above algorithm yields a strategy $q^*:[0,T]\to U,$ then $q^*$ is an optimal strategy.  
\end{prop}
\begin{proof}
Let $(\mu^*,q^*)$ be the pair which is determined through the algorithm and let $\{\tau_i\}_{i=1}^N$ and $\{\sigma_i\}_{i=1}^T$ be the associated sequences of time, so that $\mu^*$ is constant over each interval $(\tau_i,\sigma_i).$  Assume that the conditions of Proposition~\ref{prop:ps} are satisfied up until time $\tau_k\in[0,T).$  At each $i\in\{1,\ldots,N\},$ we have
\begin{align}\label{lambdamu}
 \sup\Lambda(\tau_i,\ell[q^*](\tau_i))\leq\mu^*(\tau_i^+)\leq\inf\Lambda'(\tau_i,\ell[q^*](\tau_i)).
\end{align}
  If the conditions of step 1 are satisfied at $\tau_k,$ then $q^*$ is clearly admissible and satisfies the properties of the proposition.  Hence, $q^*$ is optimal and the algorithm is complete.  

Assume therefore that the conditions of step 1 do not hold at $\tau_k$ so that $\sup\Lambda(\tau_k,m)=\inf\Lambda'(\tau_k,m).$  It is clear that the strategy defined by the algorithm is admissible, and it remains to check that conditions (\ref{eq:mucond1})-(\ref{eq:mucond3'}) are satisfied by $\mu^*.$  

To this end, note that, the algorithm ensures that the conditions are satisfied over the intervals $(\tau_k,\sigma_k)$ and $(\sigma_k,\tau_{k+1}],$ and it remains to check that the conditions hold at $\sigma_k.$  Suppose first that case a) of step 2 is satisfied at time $\tau_k,$ so that $\ell[q^*](\sigma_k)=E^+(\sigma_k).$  Then, by construction, we have 
\begin{align}\label{proof1}
\mu^*(\sigma_k^-)\in\Lambda(\sigma_k,E^+(\sigma_k)). 
\end{align}
Thus, if $\sigma_k=\tau_{k+1}$ then, together with (\ref{lambdamu}), this implies
\begin{align}\label{proof2}
\mu^*(\sigma_k^-)\leq\sup\Lambda(\sigma_k,E^+(\sigma_k))\leq\mu^*(\sigma_k^+) 
\end{align}
and condition (\ref{eq:mucond2'}) is satisfied at $\sigma_k.$  If, on the other hand, $\sigma_k<\tau_{k+1},$ then the construction of $\mu^*$ through the algorithm immediately implies that $\mu^*(\sigma_k^-)\leq\mu^*(\sigma_k^+).$  %It is always possible to find a $\mu^*$ which satisfies conditions (i) and (ii) of step 3 due to the monotonicity property of Lemma~\ref{lemma:xinc}.

Similarly, if case b) of step 2 holds at $\tau_k,$ then $\ell[q^*](\sigma_k)=E^-(\sigma_k)$ and  
\begin{align}\label{proof3}
 \mu^*(\sigma_k)\in\Lambda'(\sigma_k,E^-(\sigma_k)).
\end{align}
Thus, if $\sigma_k=\tau_{k+1}$ then, together with (\ref{lambdamu}), this implies
\begin{align}\label{proof4}
 \mu^*(\sigma_k)\geq\inf\Lambda'(\sigma_k,E^-(\sigma_k))\geq\mu^*(\sigma_k^+)
\end{align}
and condition (\ref{eq:mucond3'}) is satisfied at $\sigma_k.$
If, on the other hand, $\sigma_k<\tau_{k+1},$ then the construction of $\mu^*$ through the algorithm immediately implies that $\mu^*(\sigma_k^-)\geq\mu^*(\sigma_k^+).$  %As above, it is always possible to find a $\mu^*$ which satisfies conditions (i) and (ii) of step 3 due to the monotonicity property of Lemma~\ref{lemma:xinc}.

Finally, if case c) holds at $\tau,$ then the above two cases together imply that the conditions of Proposition~\ref{prop:ps} are satisfied.
\end{proof}

\begin{prop} \label{prop:algpot2}
If the algorithm terminates early, then there is no pair $(\mu^*,q^*)$ which satisfies the conditions of Proposition~\ref{prop:ps}. 
\end{prop}

\begin{proof}
Let $(\mu^*,q^*)$ be a pair generated by the algorithm.  (If the algorithm terminates early, at a time $T_0$ say, then $\mu^*$ and $q^*$ are both considered as functions over the domain $[0,T_0].$)  
Let $\{\tau_i\}_{i=1}^{N^*}$ and $\{\sigma_i\}_{i=1}^{N^*}$ be the sequences of time generated by the algorithm, so that $\mu^*$ is constant over each $(\tau_k,\sigma_k).$  We write $m^*_k:=\ell[q^*](\tau_k)$ % and
%\begin{align*}
% \lambda^*_k&:=\mu^*(t)=\sup\Lambda(\tau_k,m^*_k) \ \ \ \ \ \ \forall t\in(\tau_k,\sigma_k).
%\end{align*}
and assume that each interval $(\tau_k,\sigma_k)$ is maximal in the sense that if $(\tau,\sigma)\supset(\tau_k,\sigma_k)$ is a strictly larger interval, then $\mu^*$ is not constant over this interval.  

Suppose also that there is a pair $(\mu,q)$ which satisfies the conditions of Proposition~\ref{prop:ps}.  Analogously, let $\{t_i\}_{i=1}^N$ and $\{s_i\}_{i=1}^N$ be increasing sequences which define the end-times of all maximal intervals $(t_k,s_k)$ over which $\mu$ is constant.  We write $m_k:=\ell[q](t_k)$ and
\begin{align*}
\lambda_k&:=\mu(t) \ \ \ \ \ \ \forall t\in(t_k,s_k).
\end{align*}
We make the following additional assumptions on the pair $(\mu,q)$: Suppose there is a pair of times $\tau,\tau'\in[0,T]$ such that $\tau<\tau'$ and $\ell[q](t)\in\{E^-(t),E^+(t)\}$ for each $t\in\{\tau,\tau'\}.$  If there exists a pair $(\mu',q')$ such that (i) the conditions of Proposition~\ref{prop:ps} are satisfied, (ii) $\ell[q'](\tau)=\ell[q](\tau)$ and $\ell[q'](\tau')=\ell[q](\tau'),$ and (iii) $(d/dt)\mu'(t)=0$ for all $t\in(\tau,\tau'),$ then we assume that $(\mu,q)$ agrees with $(\mu',q')$ over the larger interval $[0,\tau').$  In other words, we always take $\mu$ to be constant whenever there is the choice.

Throughout this proof, we will use the notation $\lambda\sim_k\lambda'$ if the set of $(\lambda,t_k)-$admissible functions coincides exactly with the set of $(\lambda',t_k)-$admissible functions over the interval $(t_k,s_k).$  This assumption is only required to handle the degenerate cases where $q$ is identically zero over $(t_k,s_k).$  In these cases, there may be a choice of values to assign to $\mu,$ each of which would result in the same $(\lambda_k,t_k)-$admissible $q$ over $(t_k,s_k).$ %(In fact, Lemma~\ref{lemma:xinc} implies that it suffices to find only one function $x$ which is both $(\lambda,\tau_k)-$admissible and $(\lambda',\tau_k)-$admissible, in order to conclude that $\lambda\sim_k\lambda'$.)  
We introduce the ordering $\lambda\prec_k\lambda'$ if $\lambda\nsim_k\lambda'$ and $\lambda<\lambda'.$

%\update{We also make the assumption that $q^*$ is not identically zero over any of these intervals.  This assumption may be made without loss of generality since the algorithm will still set $q(t)=0$ for all $t$ which lie outside the intervals of the form $(\tau_k,\sigma_k)-$ hence the outcome of the algorithm is unchanged by the assumption.  However, the assumption is useful for the following reason: if $q^*(t)\neq 0$ and solves (\ref{minptwise}) at some $t\in[0,T],$ then there is no other value of $\mu^*(t)$ which admits $q^*(t)$ as a minimizer.  Equivalently, if $q^*$ is $(\lambda,\tau_k)-$admissible for some $\lambda\in\re,$ and if $q^*$ is not identically zero over $(\tau_k,T],$ then $q^*$ is not $(\lambda',\tau_k)-$admissible for any other $\lambda'\in\re\setminus\{\lambda\}.$} 

The proof consists of proving that the following statements are true at each $k\in\{1,\ldots,N\}$:
\begin{enumerate}
\item $\lambda_k\sim_k\sup\Lambda(t_k,m_k)$ and $q$ is $(\lambda_k,t_k)-$admissible over the interval $(t_k,s_k).$  
\item There exists $j\in\{1,\ldots,N^*\}$ such that $t_k=\tau_j.$
\item If $j\in\{1,\ldots,N^*\}$ and $t_k=\tau_j,$ then $s_k=\sigma_j$ and $\ell[q](t_k)=\ell[q](\sigma_k).$
%\item $N^*=N$ and, for each $k\in\{1,\ldots, N\},$ we have $\tau_k=t_k.$
%\item If $k\in\{1,\ldots,N\},$ then $\sigma_k=s_k$ and $\ell[q^*](\sigma_k)=\ell[q](s_k).$
\end{enumerate}
The first statement is required to prove statement 2, and implies that the algorithm determines the correct value of $\mu$ over each interval $(t_k,s_k)$.  Together, statements 2 and 3 then complete the proof, since they imply that the algorithm uncovers all intervals $(t_k,s_k)$ over which $\mu$ is constant.  Moreover, statement 3 implies that the strategy uncovered by the algorithm agrees with $q^*$ at the end point of each of these intervals.  Hence, in between each $\sigma_k$ and $\tau_{k+1},$ the strategies $q$ and $q^*$ necessarily lie on the capacity constraints and so must coincide.  In particular, the existence of the pair $(\mu,q)$ implies that the algorithm does not terminate early but instead generates a pair $(\mu^*,q^*)$ which are defined over $[0,T].$  

\textit{Proof of statement 1:}  Let $k\in\{1,\ldots,N\}.$  It is clear from the definition of $(\lambda_k,t_k)-$admissibility and from the requirement that $q$ must be a minimizer of (\ref{minptwise}), that $q$ must be $(\lambda_k,t_k)-$admissible over the interval $(t_k,s_k)$.  We need to show that $$\lambda_k\sim_k\sup\Lambda(t_k,m_k)=:\theta_k.$$

Recall that the maps $\lambda\mapsto u(t,\lambda)$ of Lemma~\ref{lemma:xinc} are piecewise continuous and monotone increasing at each time $t\in[0,T].$  These maps are constructed in such a way that one can always find a suitable $u(t,\lambda_k)$ to coincide with $q$ over $(t_k,s_k).$  Thus, the monotonicity of $u$ with respect to its second argument implies the relations:
\begin{align}
 &\beta\prec_k\theta_k \quad \Rightarrow \quad \beta\in\Lambda(t_k,m_k)\setminus\Lambda'(t_k,m_k)\label{beta1}\\
 &\beta\succ_k\theta_k \quad \Rightarrow \quad \beta\in\Lambda'(t_k,m_k)\setminus\Lambda(t_k,m_k)\label{beta2}.
\end{align}
Suppose now that  $\lambda_k\prec_k\theta_k.$  If $m_k+\ell[q](s_k)=E^+(s_k),$ then this immediately implies that $\lambda_k\sim_k\theta_k$ since, otherwise, the continuity and monotonicity of $u$ would imply the existence of $\beta$ such that $\lambda_k\prec_k\beta\prec_k\theta_k$ and a $(\beta,t_k)-$admissible $q'$ such that $m_k+\ell[q'](t)\geq E^+(t)$ for all $t\in[t_k,s_k],$ with strict inequality at $t=s_k.$  This would imply that $\beta\in\Lambda'(t_k,m_k),$ which contradicts (\ref{beta1}).  

If, on the other hand, $m_k+\ell[q](s_k)=E^-(s_k),$ then Proposition~\ref{prop:ps} requires that $\mu(s_k^+)-\mu(s_k^-)\leq 0.$  If $\mu$ is constant over an interval $(s_k,t')\subset(s_k,T],$ then the monotonicity property of Lemma~\ref{lemma:xinc} implies that the lower capacity constraint will be broken by $q,$ and no admissible solution will be found.  If $\mu$ is not constant over any such interval, then Proposition~\ref{prop:ps} requires that $q(t)=u(t,\mu(t))=(d/dt)E^-(t)+\alpha E^-(t)$ for all $t\in(s_k,t').$  By the monotonicity of $u,$ this is only possible if $(d/dt)\mu(t)>0$ at some $t\in(s_k,t')$ or if $\mu(t'^+)-\mu(t'^-)>0,$ since otherwise the fact that $\lambda_k\in\Lambda(t_k,m_k)$ implies that $q$ will break the lower capacity constraint.  However, this contradicts condition~(\ref{eq:mucond3}) of the proposition.  Hence, we must have $\lambda_k\sim_k\theta_k$ or $\lambda_k\succ_k\theta_k.$  If $\lambda_k\succ_k\theta_k,$ then similar arguments as above arrive at analogous contradictions to Proposition~\ref{prop:ps} and we conclude that we must therefore have $\lambda_k\sim_k\theta_k.$  This completes the proof of statement 1.

\textit{Proof of statement 2:}  Suppose first that $\tau=t_k$ but that there is no $j\in\{1,\ldots,N^*\}$ such that $\tau=\tau_j.$  Then, statement 1 implies that $\lambda_k\sim_k\sup\Lambda(t_k,m_k)=\theta_k$ and that $q$ must coincide with some $(t_k,\lambda_k)-$admissible $x$ over $(t_k,\sigma_k).$  However, since $\tau$ does not coincide with any $\tau_j,$ the conditions of steps 1 and 2 of the algorithm do not hold at time $\tau.$  If $x\in X(t_k,m_k),$  then Proposition~\ref{prop:ps} implies that $s_k<T$ and $\ell[q](s_k)=E^-(s_k)$ (since the conditions of step 2 do not hold).  Thus, Proposition~\ref{prop:ps} implies that either (i) $s_k=t_{k+1}$ and $\mu(s_k^+)-\mu(s_k^-)\leq 0,$ or (ii) $s_k<t_{k+1}$ and $\mu(t)-\mu(s_k^-)\leq 0,$ for all $t\in[s_k,t_{k+1}].$  In case (i), the monotonicity property of Lemma~\ref{lemma:xinc} implies that $q$ breaks the lower capacity constraint.  In case (ii), since $x$ breaks the lower capacity constraint, we must have that $x(s)<(d/ds)E^-(s)+\alpha E^-(s)$ for all $s$ in some interval $I\subset[s_k,t_{k+1}].$  However, since we require that $q(s)=(d/ds)E^-(s)+\alpha E^-(s)$ for all $s\in(s_k,t_{k+1}),$ this means that we must have $(d/ds)\mu(s)>0$ for all $s$ in some subset $I'\subset I.$  This contradicts the conditions of Proposition~\ref{prop:ps}, and we conclude that there exists $j\in\{1,\ldots,N^*\}$ such that $\tau=\tau_j.$  The same result follows from similar arguments if $x\in X'(t_k,m_k).$  This completes the proof of statement 2.

\textit{Proof of statement 3:}  Let $\tau=\tau_j=t_k.$  First notice that the proof of statement 2 yielded a contradiction when we assumed that $q$ coincided with a $(\lambda_k,t_k)-$admissible $x\in X(t_k,m_k)$ such that $s_k<T$ and $\ell[q](s_k)=E^-(s_k).$  Since statement 1 implies that $q$ must coincide with some $(\lambda_k,t_k)-$admissible $x$ over $(t_k,s_k),$ it follows that, if $x\in X(t_k,m_k),$ then either $s_k=T$ or $\ell[q](s_k)=E^+(s_k).$  Similarly, if $x\in X'(t_k,m_k),$ then either $s_k=T$ or $\ell[q](s_k)=E^-(s_k).$  The assumption that the intervals $(t_k,s_k)$ are maximal thus implies that $s_k$ coincides with $\sigma_j$ and that $\ell[q](s_k)=\ell[q^*](\sigma_j),$ provided that $\sigma_j$ is independent of any choices made at steps 1 or 2 of the algorithm.  This independence is not immediately clear since the general assumptions on the cost function $L$ mean that there may be more than one choice of $(\lambda_k,\tau)-$admissible strategy.

%We first show that, if  then $\sigma_j=s_k.$  By statement 1 and by the construction of $q^*$ through the algorithm, we know that $q$ and $q^*$ coincide with $(\lambda_k,\tau)-$admissible strategies over the intervals $(\tau,s_k)$ and $(\tau,\sigma_j)$ respectively.  However, the general assumptions on the cost function $L$ mean that there may be more than one $(\lambda_k,\tau)-$admissible strategy and therefore $q$ and $q^*$ may not agree.  In particular, this may result in the inequality $s_k\neq\sigma_j.$    We will now show that in fact we have equality: $s_k=\sigma_j.$  Since $q$ is chosen to satisfy assumption (i), it suffices to prove that $\sigma_j$ is independent of the choice of strategy at step 1 or 2 of the algorithm.

%We will show first that the algorithm is well-defined.  This reduces to proving the following: suppose that the algorithm uniquely selects the first $k-1$ times of the sequences $\{\tau_i\}_{i=1}^N$ and $\{\sigma_i\}_{i=1}^N.$  Then, by construction, $\tau_k$ is uniquely defined by the algorithm, and we need to show that the selection of $\sigma_k$ is also unique.  In other words, we must prove that the selection of $\sigma_k$ is independent of the choice of $x$ made at step 2 of the algorithm, at time $\tau_k.$

To prove the independence of $\sigma_j$ from the choice of strategy made in the algorithm at steps 1 or 2, suppose that there are two $(\lambda_k,\tau)-$admissible strategies $x^1$ and $x^2$ which satisfy the conditions of step 1 or 2 of the algorithm.  %For $i\in\{1,2\},$ let $N^i$ correspond to the $N$ of the algorithm, if $x=x^i$ is selected at time $\tau_k,$ and let $\{\tau_j^i\}_{j=k}^{N^i}$ and $\{\sigma_j^i\}_{j=k}^{N^i}$ be the corresponding sequences of times.  Then,the final time in $(\tau_k,s[x_i](\tau_k,m))$ such that
%\begin{align*}
% m^i:=m+\ell[x^i](\sigma_k^i)=\left\{\begin{array}{lll}
%         E^+(\sigma_k^i) & & \textnormal{if} \ x^i\in X(\tau_k,m)\\
%         E^-(\sigma_k^i) & & \textnormal{if} \ x^i\in X'(\tau_k,m)
%                             \end{array}
%\right.
%\end{align*}
Let $\sigma^i$ denote the time $\sigma_j$ if strategy $x^i$ is chosen at time $\tau$ and assume without loss of generality that $\sigma^1\leq\sigma^2.$  The monotonicity property of Lemma~\ref{lemma:xinc} makes it immediately clear that, if we choose $x^1$ then the algorithm will select $\mu^*(t)=\lambda_k$ for all $t\in(\sigma^1,\sigma^2).$  Hence, $\mu^*$ is constant over $(\tau,\sigma^2).$  In particular, if there are several choices of strategy to choose at time $\tau,$ and if $\sigma^2$ is the latest time $\sigma_j$ associated with any of the choices, then the assumption that each interval $(\tau_j,\sigma_j)$ is maximal implies that $\sigma_j=\sigma^2,$ regardless of the strategy chosen.  This completes the proof of statement 3.

%\update{Finally, suppose that $\ell[q^*](\sigma_k)\neq\ell[q](s_k).$  If $\ell[q^*](\sigma_k)=E^+(\sigma_k),$ then the capacity constraints require that $q^*(t)<q(t)$ for all $t$ in some subinterval $I$}

%\update{then Proposition~\ref{prop:ps} requires that $\mu^*(\sigma_k)\geq\sup\Lambda(\tau,m_k)\geq\mu(\sigma_k),$ whilst the capacity constraints require that $q^*(\sigma_k)\leq q(\sigma_k).$  This contradicts the monotonicity property of Lemma~\ref{lemma:xinc} and we conclude that $\ell[q^*](\sigma_k)=\ell[q](s_k).$}

\end{proof}

\section{A simple storage model}\label{sec:constraints}

%\subsection{A simple storage model}\label{sec:example}

We introduce here a simple but practically relevant model of a store, which operates purely within a single wholesale market.  The set of admissible powers $U$ is taken to be of the form $U=[-q_{\max}^-,q_{\max}^+],$
%  \begin{align}\label{eq:U}
%  U=[-q_{\max}^-,q_{\max}^+],
%  \end{align}
  with $-q_{\max}^-<0<q_{\max}^+.$  \lff{The store is assumed to have a fixed maximum energy capacity $M>0,$ so that the admissible domain satisfies} $E^-(t)=0$ and $E^+(t)=M$ for all $t\in(0,T).$ We further assume that we require the store to be empty at both the start and end times, so that $E^-(0)=E^+(0)=E^-(T)=E^+(T)=0.$  The cost functional $C$ is defined by
\begin{align}\label{eq:Cexample}
\lf{C[q]=\int_0^Tw(q(t))p(t)q(t),}
\end{align}
\lfff{where} $p(t)\in\re$ \lfff{is the price for one unit of power at time} $t.$  Here, $w:U\to(0,\infty)$ is a piecewise continuous function which acts as a weighting to the cost of buying and selling electricity, and should satisfy that $w(x)\geq 1$ if $x\geq 0;$ and $0<w(x)\leq 1$ if $x\leq 0.$  More intuitively, $w$ can be thought of as the result of inefficiencies during the charging and discharging processes of the store: in order to fill the store at a rate $q(t)$ at time $t,$ the storage operator would actually need to purchase a greater amount of electricity, since power is lost during the charging process.  Similarly, whenever an amount $q(t)$ is discharged from the store, a lesser amount is actually available to sell on the market.  For simplicity, we assume that $w$ takes on two values $w_1\geq 1$ and $w_2\in(0,1],$ so that
\begin{align}\label{w}
 w(x)=\left\{\begin{array}{lll}
              w_1 & & \textnormal{if} \ x\geq 0\\
              w_2 & & \textnormal{if} \ x<0.
             \end{array}
\right.
\end{align}
In the special case where there are no additional operating constraints (i.e. $\Gamma=\emptyset)$, this cost functional is convex and thus an optimal strategy exists.  Proposition~\ref{prop:algpot2} states moreover that an optimal strategy can be determined via the algorithm of Section~\ref{sec:alg}.  \lfff{Some simple results are immediately available, as stated in the following lemma.}

% With a cost functional of this form, 
%A simple result is immediately available, as stated in the following lemma: if the price of electricity is constant over all time, then it is not worth investing in storage.
\begin{lemma}
\lf{Define the parameters of a store as above and let the cost functional $C$ be given by (\ref{eq:Cexample}).  \lfff{Assume that} $p(t)\geq 0$ \lfff{for all} $t\in[0,T]$ and let $\ell_0,\ell_T\geq 0$ be the required levels of stored energy at times $0$ and $T$ respectively.}  \lfff{The following properties hold:} 
\begin{enumerate}
 \item If $\ell_T\geq\ell_0$ and $\dot p\equiv 0,$ then there is no non-zero strategy $q\in X$ which yields a positive profit.
 \item \lfff{The minimum cost, over all admissible strategies, is a decreasing function of} $w_2,M,q_{\max}^-$ and $q_{\max}^+.$  \lfff{It is an increasing function of} $\alpha$ and $w_1.$
\end{enumerate}
\end{lemma}
\begin{proof}
If $q\in X$ is not the zero function, and if $p(t)=p$ is constant for all $t\in[0,T],$ then
\begin{align*}
 C[q]= p\int_0^Tw(q(t))q(t) \ dt\geq p\int_0^Tq(t) \ dt=p\left(\ell_T-\ell_0\right)+\alpha\int_0^T\ell[q](t) \ dt \geq 0.
\end{align*}
This proves statement 1.  To prove statement 2, let $q^*$ be the optimal strategy associated with the parameters $\theta=(\alpha,w_1,w_2,M,q_{\max}^-,q_{\max}^+)$ and write $V(\theta)~=~C[q^*].$
It is easy to see that $q^*$ is also admissible with respect to the parameters $\theta+\delta\gamma,$ for any $\delta>0$ and any $\gamma=(\gamma_1,\ldots, \gamma_6)$ with $\gamma_1=0,$ \ $\gamma_2\in\{-1,0\}$ and $\gamma_3,\ldots,\gamma_6\in\{0,1\}.$  Thus, if $w_{\gamma}(x)$ is defined by $w_{\gamma}(x)=\theta_2+\delta\gamma_2$ whenever $x\geq 0$ and $w_{\gamma}(x)=\theta_3+\delta\gamma_3$ whenever $x<0,$ then 
\begin{align*}
V(\theta+\delta\gamma)\leq \int_0^Tw_{\gamma}(q^*(t))p(t)q^*(t) \ dt\leq \int_0^Tw(q^*(t))p(t)q^*(t) \ dt=V(\theta).
\end{align*}
This proves the required growth conditions of the minimum cost with respect to $w_1,w_2,M,q_{\max}^-$ and $q_{\max^+}.$  To complete the proof, let $\alpha'<\alpha$ and denote by $\ell'[q](t)$ the level of stored energy at time $t$ associated with the strategy $q$ and the leakage rate $\alpha'.$  Set $q'(t)=q^*(t)$ whenever $q^*(t)\leq 0$ or $\ell'[q'](t)<M.$ For all other $t\in[0,T],$ set $q'(t)=\min(\alpha M,q^*(t)).$  Then $q'$ is admissible with respect to the new parameter vector $\theta'=\theta+(\alpha'-\alpha,0,0,0,0,0).$  Hence,
\begin{align*}
 V(\theta)=C[q^*]=C[q']+\int_0^Tw(q(t))p(t)\left(q^*(t)-q'(t)\right) \ dt\geq C[q']\geq V(\theta')
\end{align*}
and the minimum cost function is increasing with respect to the leakage rate $\alpha.$

%$q'(t)=q^*(t)-(\alpha-\alpha')\ell[q^*](t)$ for all $t\in[0,T].$  Then equation~(\ref{elldefn}) implies that $q'$ is admissible with respect to the new parameter vector $\theta'=\theta+(\alpha'-\alpha,0,0,0,0,0).$  Hence,
%\begin{align*}
% V(\theta)=C[q^*]=C[q']+(\alpha-\alpha')\int_0^Tw(q'(t))p(t)\ell[q^*](t) \ dt\geq C[q']\geq V(\theta').
%\end{align*}
%In particular, the minimum cost function is increasing with respect to the leakage rate $\alpha.$
\end{proof}

\subsection{A storage model with minimum switching times}\label{sec:switching}
\lf{Throughout this section, we employ the notation $t^+$ and $t^-$ to mean $\max\{t,0\}$ and $\min\{t,T\}$ respectively for any $t\in\re.$}  We \lff{investigate the impact of introducing a non-empty set of additional constraints $\Gamma$ to the simple storage model introduced above.  In this setting, Lemma~\ref{lemma:constraints} no longer holds in general.  It is therefore not clear that there will exist a method with similar time-localization properties as the algorithm of Section~\ref{sec:alg}.  We demonstrate this problem explicitly in this section.  In particular, we} impose a minimum switching time constraint $\tau>0$ on the simple storage model introduced above.  This is the minimum amount of time for which the store must be idle when switching between the charging and discharging modes of operation.  \lf{Hence the additional constraint set $\Gamma$ consists of the requirement that any power output $q$ must satisfy that}
\begin{align}\label{Gamma}
 q(t)q(s)\geq 0 \ \ \ \ \ \ \forall t\in[0,T] \ \textnormal{and} \ s\in(t-\tau,t+\tau)\cap[0,T].
\end{align}
\lf{The minimization problem of Proposition~\ref{prop:ps} is difficult to solve in its current form, but we may rewrite the problem by adapting} the Lagrangian $\mathcal L$ of (\ref{relaxed}) \lf{to incorporate the minimum switching time constraint as follows:}
$$\Lambda[q,\mu,\lambda]:=\underbrace{\int_0^T\Big(w(q(t))p(t)-e^{\alpha t}\mu(t)\Big)q(t) \ dt}_{\lf{=\mathcal L[q,\mu]}}-\int_0^T\int_t^{(t+\tau)^-}\lambda(t,s)q(s)q(t) \ ds  \ dt$$ 
for all piecewise continuous $q,\mu:[0,T]\to U$ and $\lambda:[0,T]^2\to[0,\infty].$  \lf{In order to apply Proposition~\ref{prop:ps}, we need to reduce $\Lambda$ into a form which is compatible with (\ref{eq:propmax}).}  To do this, we carefully choose $\lambda$ to depend on \lf{$q$ and} $\mu$ as follows: For a given $\mu,$ let $q_{\mu}:[0,T]\to U$ solve
%\begin{align}\label{eq:qmu}
% \lf{\mathcal L[q_{\mu},\mu]\leq\mathcal L[q,\mu]}
%\end{align}
%\lf{for all piecewise continuous $q:[0,T]\to U.$  Since the cost functional $C$ is of the form (\ref{eq:Cexample}), with each element of the integrand dependent only on $t$ and the power output at that time, this is equivalent to solving}
\begin{align}\label{eq:qmu}
\Big(w(q_{\mu}(t))p(t)-e^{\alpha t}\mu(t)\Big)q_{\mu}(t)\leq \Big(w(x)p(t)-e^{\alpha t}\mu(t)\Big) x \ \ \ \ \ \forall x\in U.
\end{align}
Since the efficiency function $w$ takes the simple form (\ref{w}), it is clear that $q_{\mu}(t)$ takes a value in $\{-q_{\max}^-,0,q_{\max}^+\}$ whenever $w_1p(t)-e^{\alpha t}\mu(t)\neq 0$ and $w_2p(t)-e^{\alpha t}\mu(t)\neq 0$ at $t\in[0,T].$  \lf{In these cases $q_{\mu}(t)$ is uniquely defined.}  If one of these relations is an equality instead, however, then $q_{\mu}(t)$ may take any value in $U.$  Hence, there may be multiple solutions associated with $\mu.$  Let now $\lambda$ minimize
\begin{align}\label{lambdamin}
 \int_0^T\left|w(q_{\mu}(t)p(t)-e^{\alpha t}\mu(t)-\int_{(t-\tau)^+}^{(t+\tau)^-}\lambda(t,s)q_{\mu}(s)\right|dt
 \end{align}
over all $\lambda:[0,T]^2\to[0,\infty]$ such that (i) $\lambda(t,s)=0$ whenever $q_{\mu}(s)q_{\mu}(t)\geq 0$ and (ii) the inequality 
\begin{align}\label{lambdamin2}
 \Big(w(q_{\mu}(t))p(t)-e^{\alpha t}\mu(t)-\int_{(t-\tau)^+}^{(t+\tau)^-}\lambda(t,s)q_{\mu}(s)\Big)\left(w(q_{\mu}(t))p(t)-e^{\alpha t}\mu(t)\right)\geq 0 %\ \ \ \ \forall t\in[0,T]
\end{align}
holds for all $t\in[0,T].$  We write $\lambda=\lambda(q_{\mu}).$  

\lf{The significance of the above construction is that an optimal strategy can now be found by searching for minimizers of $\Lambda.$  Specifically, if  a strategy $q^*$ satisfies that
\begin{align}\label{Lambdamin}
 \Lambda[q^*,\mu^*,\lambda^*]\leq\Lambda[q,\mu^*,\lambda^*]
\end{align}
for all piecewise continuous $q:[0,T]\to U,$ and if (i) $(q^*,\mu^*)$ satisfies the conditions~(\ref{eq:mucond1})--(\ref{eq:mucond3'}), (ii) $\lambda^*$ is of the form $\lambda^*=\lambda(q_{\mu^*})$ and (iii) $q^*$ is admissible, then $q^*$ is an optimal strategy.  To see this, choose $\rho^*$ large enough that the relation
\begin{align}\label{rho}
\rho^*\geq q_{\max}^+q_{\max}^-\int_0^T\int_t^{(t+\tau)^-}\lambda^*(t,s) \ ds \ dt
\end{align}
holds.  It is clear from the assumption that $q^*$ is admissible that $S[q^*]=0$ and it follows from (\ref{Lambdamin}) that $q^*$ minimizes $\mathcal L[q,\mu^*]+\rho^*S[q]$ over all admissible $q\in X.$  Assumption~(\ref{rho}) moreover implies that $q^*$ minimizes $\mathcal L[q,\mu^*]+\rho^*S[q]$ over all piecewise continuous $q:[0,T]\to U$ because, for any non-admissible $q,$ we have
\begin{align*}                                                                                                                                                                                                                                                                                                                                                                                                                                            \mathcal L[q,\mu^*]+\rho^*\underbrace{S[q]}_{=1}%=\mathcal L[q,\mu^*]+\rho^*
&\geq\mathcal L[q,\mu^*]-\int_0^T\int_t^{(t+\tau)^-}\lambda^*(t,s)q(s)q(t) \ ds \ dt%\\   &\geq\mathcal L[q^*,\mu^*]
=\mathcal L[q^*,\mu^*]+\rho^* \underbrace{S[q^*]}_{=0}.                                                                                                                                                                                                                                                                                                                                                                                                                                          \end{align*}
Thus, Proposition~\ref{prop:ps} implies that $q^*$ is optimal.}

\lf{Condition (iii) of the above argument requires that the minimizer $q^*$ of (\ref{Lambdamin}) is admissible, in order for $q^*$ to be an optimal strategy.  As in the algorithm of Section~\ref{sec:alg}, a suitable choice of $\mu^*$ will ensure that the capacity constraints are satisfied, whilst the power rating constraints are immediately absorbed into the minimization problem~(\ref{Lambdamin}).  The following lemma ensures that, as long as a minimizer $q^*$ exists, then we may select $q^*$ so that the switching time constraints are also satisfied.}

\begin{lemma}\label{lemma:lambda}
\lf{Let $\mu:[0,T]\to\re$ be piecewise differentiable and let $\lambda$ be of the form $\lambda=\lambda(q_{\mu})$ for some minimizer $q_{\mu}$ of (\ref{eq:qmu}).}  Let $Y$ be the set of piecewise continuous $q^*:[0,T]\to U$ such that 
%\begin{align*}
$\Lambda[q^*,\mu,\lambda]\leq \Lambda[q,\mu,\lambda]$
%\end{align*}
for all piecewise continuous $q:[0,T]\to U.$  If $Y\neq\emptyset,$ then there exists $q^*\in Y$ which satisfies the switching time constraints~(\ref{Gamma}).   
 \end{lemma}

\begin{proof}
 Suppose the claim is not true.  Then there exists $t,s\in[0,T]$ such that $|t-s|<\tau$ and $q^*(t)q^*(s)<0.$  Thus $\lambda(t,s)q^*(t)q^*(s)\leq 0.$  Note that
 \begin{align*}
  f[q,\mu,\lambda](t):=\frac{\delta \Lambda}{\delta q(t)}[q,\mu,\lambda]=w(q((t))p(t)-e^{\alpha t}\mu(t)-\int_{(t-\tau)^+}^{(t+\tau)^-}\lambda(t,s)q(s) \ ds,
 \end{align*}
\lf{where $\delta\Lambda/\delta q(t)$ denotes the functional derivative of $\Lambda$ with respect to $q(t)$.}  In particular, the term $\lambda(t,s)$ only influences the value of \lf{$q^*$ at times $t$ and $s.$}  The condition that $\lambda$ minimizes (\ref{lambdamin}) subject to the constraint (\ref{lambdamin2}) implies that $f[q^*,\mu,\lambda](t)=0$ or $f[q^*,\mu,\lambda](s)=0.$  In particular, we may always choose a $q^*$ so that $q^*(t)=0$ or $q*(s)=0.$
\end{proof}

\lf{An interpretation of the role of $\lambda^*$ in the minimization problem~(\ref{Lambdamin}) is that the function acts to adjust $q_{\mu^*}$ to a new strategy which satisfies the switching time constraints.  It does this by finding all pairs of times $(t,s)$ such that $|s-t|<\tau$ and $q_{\mu^*}(s)q_{\mu^*}(t)<0,$ and adjusting $q_{\mu^*}(t)$ or $q_{\mu^*}(s)$ to 0.  The condition that $\lambda^*(t,s)=0$ whenever $q_{\mu^*}(t)q_{\mu^*}(s)\geq 0$ ensures $q_{\mu^*}$ is only adjusted at times when the minimum switching time constraint~(\ref{Gamma}) is broken.  Choosing $\lambda^*$ to minimize (\ref{lambdamin}) subject to the constraint~(\ref{lambdamin2}) ensures that $\lambda^*$ makes these adjustments in such a way that, if $a(t)$ is the absolute change in instantaneous cost from the original strategy $q_{\mu^*}$ at time $t,$ then the accumulation of these cost changes $\int_0^Ta(t)dt$ is minimal.  Thus, an equivalent and generally easier method for determining the optimal strategy $q^*,$ given $\mu^*,$ is to find a suitable subset of times $I\subset[0,T]$ and set $q^*(t)=0$ if $t\in I$ and $q_{\mu^*}(t)=q_{\mu^*}(t)$ if $t\in[0,T]\setminus I.$  The subset $I$ should be chosen to minimize the associated absolute cost change integral $\int_0^Ta(t)dt.$}

\lf{If $\mu^*$ is given, then the minimization problem~(\ref{Lambdamin}) is much easier to solve than (\ref{eq:propmax}).  The main task in determining an optimal strategy $q^*,$ therefore, is to choose the correct reference price function $\mu^*.$  As alluded to at the beginning of Section~\ref{sec:constraints}, this is a more difficult task now that the set of additional constraints $\Gamma$ is non-empty.  To illustrate the complexity of the problem, consider the following scenario: }

\lf{Let the power ratings and capacity constraints of the store be given by $q_{\max}^-=q_{\max}^-=1$ and $M=2$ and let the minimum switching time of the store be $\tau=1.$  Assume that there is no leakage and the store is perfectly efficient so that $\alpha=0$ and $w_1=w_2=1.$  Suppose that $T=3$ and that the price of electricity $p$ is a function of time given by $p(t)=p_i$ for all $t\in[i-1,i)$ and for each $i\in\{1,2,3\},$ where the values $p_i$ satisfy $p_3<p_1<p_2.$  Suppose further that $\half(p_1+p_2)>p_3.$  It is not difficult to see that the optimal strategy $q^*$ satisfies $q^*(t)=1$ if $t\in[0,\half]$ and $q^*(t)=-1$ if $t\in[\frac{3}{2},2],$ with $q^*(t)=0$ elsewhere.  The corresponding reference price $\mu^*$ satisfies $\mu^*(t)=\mu_1:=\half(p_1+p_2)$ if $t\in[0,2)$ and $\mu^*(t)=\mu_2=p_3$ otherwise.
%\begin{align*}
% q^*(t)=\left\{\begin{array}{lll}
%    1 & & \textnormal{if} \ t\in[0,\half)\\
%    0 & & \textnormal{if} \ t\in[\half,\frac{3}{2}]\cup [2,3]\\
%    -1 & & \textnormal{if} \ t\in[\frac{3}{2},2)\\
%               \end{array}
%\right. \ \textnormal{and} \  \mu^*(t)=\left\{\begin{array}{lll}\mu_1:=\half(p_1+p_2) & & \textnormal{if} \ t\in[0,2)\\
%\mu_2:= p_3 & & \textnormal{if} \ t\in[2,3].
% \end{array}\right.
%\end{align*}
However, the correct choice of $\mu_1$ here relies on knowing the correct choice of $\mu_2;$ unlike in the algorithm of Section~\ref{sec:alg}, the constant sections of the reference price function cannot be determined independently.  To see this, we simply remark that, if we had guessed at a constant reference price function $\mu_0^*\equiv\mu_1,$ then we would arrive at the corresponding solution $q_0^*$ with $q_0^*(t)=1$ if $t\in[0,1)\cup[2,3]$ and $q_0^*(t)=0$ otherwise.
%\begin{align*}
% q_0^*(t)=\left\{\begin{array}{lll}
%    1 & & \textnormal{if} \ t\in[0,1)\\
%    0 & & \textnormal{if} \ t\in[1,2)\\
%    1 & & \textnormal{if} \ t\in[2,3].
%               \end{array}
%\right.\end{align*}
This strategy never reaches the upper or lower capacity constraint after the initial time, and so there is no chance to update our reference price function in accordance with Proposition~\ref{prop:ps}.  Even though our guess $\mu^*_0$ is correct over the time interval $[0,2),$ it is only possible to know this when considering the value that $\mu^*$ attains over the time interval $[2,3].$}

\lf{This issue prevents the methods of Section~\ref{sec:alg} from working here.  However,} as long as the store is required to perform at least one full charge and discharge cycle then the determination of a suitable $\mu^*$ still reduces to a localized problem thanks to Lemma~\ref{lemma:constraints}.  %\lf{In particular, we reduce the final time $T$ to an earlier time $T_0$ and proceed to construct the optimal triple $(q_0^*,\mu_0^*,\lambda_0^*)$ as outlined above, subject to the condition $\ell[q_0^*](T_0)=0.$  If there is some time $t_0\in(0,T_0)$ such that $\ell[q_0^*](t_0)=M,$ then Lemma~\ref{lemma:constraints} implies that the optimal strategy agrees with $q_0^*$ over the interval $[0,t_0].$  The method then proceeds in a similar way at time $t_0,$ by replacing the initial time 0 with $t_0,$ replacing $T_0$ with a new time $T_1$ and looking for an optimal triple $(q_1^*,\mu_1^*,\lambda_1^*)$ associated with the time interval $[t_0,T_1],$ subject to the condition $\ell[q_1^*](T_1)=M.$  If there is some time $t_1\in(t_0,T_1)$ such that $\ell[q_1^*](t_1)=0,$ then the optimal strategy agrees with $q_1^*$ over the interval $[t_0,t_1].$}
Finally, we remark that, if $\mu^*$ is a reference price function and if $q^*_{\mu^*}$ is the associated strategy, constructed via the minimization~(\ref{Lambdamin}), then it can easily be shown that at each time $t\in[0,T],$ the level of the store $\ell[q_{\mu^*}^*]$ is increasing in $\mu^*.$  Whilst the algorithm of Section~\ref{sec:alg} no longer works in this more complicated setting, we can still use this monotonicity in order to reduce the choice of $\mu^*$ at each time to a suitable range, using similar arguments as in the case with no additional constraints.   %\begin{enumerate}
% \item Show that the minimiser $q$ of $\Lambda[q,\mu,\lambda(q_{\mu})]$ is admissible.
% \item Show that the problem is equivalent to the Proposition:
% $$S_{\mu^*}[q]=\int_t^{t+\tau}...$$  Adapt wording of proposition to $S$ is any function which is positive is and only if a constraint is broken.
% \item Then the problem reduces to a localised one.  Strategy of store in the midst of a full charge or discharge is determined by a single value $\mu_1$ for $\mu.$  That choice of $\mu$ is determined only by the value of $\mu$ - say $\mu_2$ - over its next full charge or discharge.  It does not depend on any subsequent values of $\mu.$  Hence, one seeks a pair $(\mu_1,\mu_2)$ whose corresponding strategies give rise to a full charge and discharge.  This then determines $\mu_1.$
% \item  Show that $\ell$ increases with $\mu.$  
%\end{enumerate}
\section{Illustrative results}
The following graphs plot the optimal strategies for a store \lf{which follows the model set out in Section~\ref{sec:constraints},} with $q_{\max}^+=q_{\max}^-$ and $M/q_{\max}^+=7.$  We take $\alpha=0,$ \ $w_1=1$ and vary $w_2$ as labeled below.  The bottom two figures plot the level of the store $\ell[q^*](t)$ at each time $t$, assuming that the optimal strategy $q^*$ is followed.  The top figures both plot the price of electricity $p:[0,T]\to[0,\infty)$ which the store faces in making its operating decision.  Here, we have chosen hourly prices from the first 230 hours of N2EX's day-ahead auction in November 2013.
\begin{figure}[H]
 \begin{center}
   \includegraphics[width=15cm,height=6cm, trim=0cm 6cm 0cm 5cm]{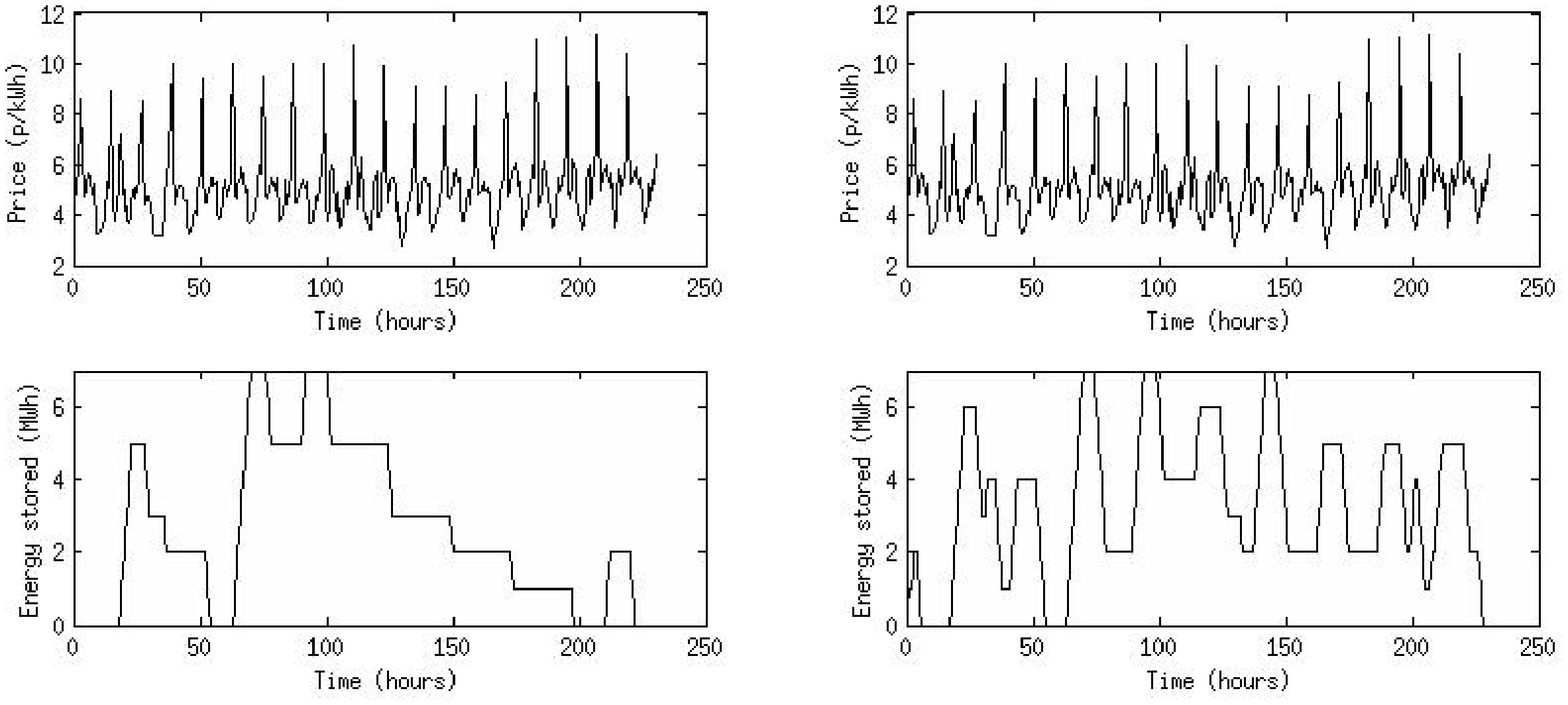}
 \end{center}
 \caption{November 2013 hourly day-ahead auction prices (top left and top right); optimal fill level of the store with $w_2=0.5$ (bottom left) and $w_2=0.7$ (bottom right).}
\end{figure}
An immediate observation is that the store cycles more frequently as its efficiency $w_2$ increases.  \lfff{Less energy is lost during operation at higher efficiencies and the store is able to earn revenues over a larger proportion of elapsed time.}  %As the efficiency increases, then the system cycles more often and is earning revenue over a larger proportion of elapsed time.  However, this proportion of time will not increase linearly with efficiency: if the store charges to full capacity and then immediately discharges, for example, then increasing the efficiency by a small amount will not change this behaviour.
We highlight here that, although the price of electricity follows a reasonably periodic pattern (with periods of roughly a day in length), it does not follow that the level of the store is the same at the start and end of each day \lff{(see, for example, \cite{MacSliSta} for a proof of the potential non-periodicity of solutions in a periodic setting)}.  This underlies an important motivation for choosing this method over a dynamic programming approach.  Using a dynamic programming method, one works backwards in time and evaluates, at each $t\in[0,T],$ the value function
\begin{align*}
 V(t,x):=\min_{q\in X}\int_t^TL(q(s),s)ds,
\end{align*}
where the minimization is subject to the constraints $q(t)=x.$  However, this requires information about all prices over the time $[0,T].$  If $T$ is large (say, $T=40$ years, the average life of a pumped hydro store), then the number of calculations becomes infeasible.  One option is to split $[0,T]$ into a union of smaller disjoint intervals $I_1\cup\ldots\cup I_N=[0,T],$ for some $N\in\n,$ and to find the optimal strategy over each of these smaller intervals.  However, this requires setting an end-state for each interval, which may not be known.  A reasonable guess in the above examples would be to assume that the store is empty at some off-peak time during each night but, as seen, such an assumption would lead to a sub-optimal solution.  Using our approach of Section~\ref{sec:alg}, the algorithm has the property that it implicitly reduces our problem to a series of new optimization problems, which are localized in time.  %More precisely, if the store is either empty or full at a time $\tau\in[0,T],$ then there is a time $\sigma\in(\tau,T]$ such that, as long as we know the restriction of $p$ over $(\sigma,\tau),$ then we can determine the optimal strategy over this interval.

\section{Conclusions}
We have presented a method for determining the operating strategy for a store to maximize its arbitrage profits.  % \lf{when the store is constrained by its power ratings and capacity limits.}  
Our setting allows for leakage, inefficiencies and general operating costs which are functions of the power output.  \lff{This setting also allows for time-varying constraints on the power output of the store.}  A significant benefit associated with this method is the implicit localization in time of the solution.  \lf{Moreover, we have shown that there only exists an optimal strategy of the form of Proposition~\ref{prop:ps} if the algorithm does not terminate early.}  \lff{As an extension, we have discussed the inclusion of more general operating constraints and proposed a method for determining optimal strategies in the presence of minimum switching time constraints.}

We believe that the theory put forward in this paper serves as a good starting point for evaluating the profits available to a store.  The assumption that prices can be predicted over suitable periods of time is a good approximation to the situation where a store trades through bilateral contracts, or through an auctioning market.  A next step would be to consider the store as a larger player, whose actions have an impact on the price of electricity, and it is believed that the approaches given in this paper can be extended to this new setting.

%Finally, we should comment that in reality, it is likely that a store will need to secure revenue through a variety of markets if it is to be a profitable investment.  In particular, it should not be ignored that a store could provide valuable balancing services to the system operator.  In such a setting, we would need to introduce some element of uncertainty into our problem and, as such, new methods would probably be required.  % 
 
 \section*{Acknowledgments}
 This work is made possible by the IMAGES (Integrated Market-fit Affordable Grid-Scale Energy Storage) research group.  The authors would like to thank the other members of the group for their suggestions and insight.  Particular thanks go to Monica Giulietti, Jihong Wang, Xing Luo, Andrew Pimm and Seamus Garvey.  The authors would also like to thank Stan Zachary, James Cruise and Richard Gibbens for their many helpful discussions related to this topic.


\begin{thebibliography}{1}
\bibitem{DECC} \textit{Planning our electric future: a White Paper for secure, affordable and low-carbon electricity,} Department of Energy and Climate Change, 2011

\bibitem{NG} \textit{UK Future Energy Scenarios,} National Grid, 2014

\bibitem{ERP} \textit{The future role for energy storage in the UK: main report,} Energy Research Partnership, 2011

\bibitem{Mac} D.J.C~MacKay: \textit{Sustainable Energy - without the hot air,}  UIT Cambridge, 2009

\bibitem{AEA} \textit{Energy Storage and Management Study,} AEA, 2010

\bibitem{Poy} \textit{Options for low-carbon power sector flexibility to 2050 - a report to the Committee on Climate Change,} P\"oyry, 2010

\bibitem{BlaStr} M.~Black, G.~Strbac: \textit{Value of Bulk Energy Storage for Managing Wind Power Fluctuations,} IEEE Transactions on Energy Conversion, 22(1) (2007), pp197--205 

\bibitem{Str} M.~Aunedi, N.~Brandon, D.~Jackravut, D.~Predrag, D.~Pujianto, R.~Sansom, G.~Strbac, A.~Sturt, F.~Teng, V.~Yufit: \textit{Strategic Assessment of the Role and Value of Energy Storage Systems in the UK Low Carbon Energy Future - Report for Carbon Trust,} 2012

\bibitem{BarInf} J.P.~Barton, D.G.~Infield: \textit{Energy Storage and Its Use With Intermittent Renewable Energy,} IEEE Transactions on Energy Conversion, 19(2) (2004), pp441--448 

\bibitem{Gru} P.~Gr\"unewald, T.~Cockerill., M.~Contestabile., P.~Pearson.: \textit{The role of large scale storage in a GB low carbon energy future: Issues and policy challenges,} Energy Policy, 39 (2011), pp4807-4815 

\bibitem{Ari} \textit{Study of Compressed Air Energy Storage with Grid and Photovoltaic Energy Generation, Draft Final Report - for Arizona Public Service Company,} Arizona Research Institute for Solar Energy (AZTISE), 2010

\bibitem{LohMin} N.~L\"ohndorf, S.~Minner: \textit{Optimal day-ahead trading and storage of renewable energies - an approximate dynamic programming approach,} Energy Syst., 1 (2010), pp61--77 

%\bibitem{UKPN} \textit{Smarter Network Storage, Six Monthly Report, December 2013,} UK Power Networks (February 2014)

\bibitem{CFGZ} J.~Cruise, L.~Flatley, R.~Gibbens, S.~Zachary: \textit{Optimal control of storage for arbitrage, with applications to energy systems} (2014) arXiv:1307.0800v1; under consideration with MOR 


\bibitem{PimGar}  A.J.~Pimm, S.D.~Garvey: \textit{The economics of hybrid energy storage plant}, International Journal of Environmental Studies, (2014), pp1-9 

\bibitem{PimGarKan}  A.J.~Pimm, A.D.~Garvey, B.K.~Kantharaj: \textit{Economic analysis of a hybrid energy storage system based on liquid air and compressed air,} (2014) Under consideration with Energy Economics.

\bibitem{MacSliSta} R.S.~MacKay, S.~Slijep\u{c}evi\'{c}, J.~Stark: \textit{Optimal scheduling in a periodic environment,}  Nonlinearity, (2000) 13, pp257--297 

%\bibitem{GoMiSa} {\sc M. Goossens, F. Mittelbach, and A. Samarin},
%{\em The} \LaTeX\ {\em Companion}, Addison-Wesley, Reading, MA, 1994.

%\bibitem{Higham} {\sc N.~J. Higham}, {\em Handbook of Writing for
%the Mathematical Sciences}, Society for Industrial and Applied
%Mathematics, Philadelphia, PA, 1993.

%\bibitem{Lamport} {\sc L. Lamport}, \LaTeX: {\em A Document
%Preparation System}, Addison-Wesley, Reading, MA, 1986.

%\bibitem{SerLev} {\sc R. Seroul and S. Levy}, {\em A
%Beginner's Book of} \TeX, Springer-Verlag, Berlin, New
%York, 1991.
\end{thebibliography}
\end{document}